\documentclass[twoside,9pt]{amsart}
\usepackage{amsmath}
\usepackage{amssymb}
\usepackage{latexsym, enumerate}
\usepackage{verbatim}
\usepackage{graphics}
\usepackage{graphicx}
\usepackage{mathrsfs}

\usepackage[all]{xy}

\newtheorem{theorem}{Theorem}
\newtheorem{corollary}{Corollary}
\newtheorem*{mainA}{Theorem A}
\newtheorem*{mainB}{Theorem B}
\newtheorem{lemma}{Lemma}
\newtheorem{proposition}{Proposition}

\newtheorem{definition}{Definition}
\newtheorem{remark}{Remark}
\newtheorem{example}{Example}

\def\RR{{\mathbb R}}
\def\ZZ{{\mathbb Z}}

\def\CC{{\mathbb C}}
\def\NN{{\mathbb Z}^+}

\def\G{{\bf G}}

\def\FF{{\mathbb F}}

\def\ZZ{{\mathbb Z}}
\def\P{{\mathcal P}}

\def\M{{\mathcal M}}

\def\G{{\mathcal G}}

\def\Q{{\mathcal Q}}

\def\inf{{\rm inf }}

\begin{document}

\title {Invariant measures and orbit equivalence for generalized Toeplitz subshifts.}

\author{Mar\'{\i}a Isabel Cortez, Samuel Petite}
\address{Departamento de Matem\'atica y Ciencia de la Computaci\'on, Universidad de Santiago de Chile,
Av. Libertador Bernardo O'Higgins 3363, Santiago, Chile.
 Laboratoire Ami\'enois de Math\'ematique Fondamentale et appliqu\'ee, CNRS U.M.R. 6140,
Universit\'e de Picardie Jules Verne, 33 rue Saint-Leu  8039
Amiens  Cedex 1  France.} \email{maria.cortez@usach.cl,
samuel.petite@u-picardie.fr}
\thanks{M. I. Cortez acknowledges financial support from proyecto Fondecyt
1100318. This work is part of the project PICS ``dynamics and combinatorics''}


\maketitle{}
\begin{abstract}
We show that for every metrizable Choquet simplex $K$ and for
every group $G$, which is infinite, countable, amenable and
residually finite, there exists a Toeplitz $G$-subshift whose set
of shift-invariant probability measures is affine homeomorphic to
$K$. Furthermore, we get that for every integer $d> 1$ and
every Toeplitz flow $(X,T)$, there exists a   Toeplitz $\ZZ^d$-subshift which
is topologically orbit equivalent to $(X,T)$.
\end{abstract}
\section{Introduction}

The {\it Toeplitz subshifts} are a rich class of symbolic systems
introduced by Jacobs and Keane in \cite{JK}, in the context of
$\ZZ$-actions. Since then, they have been extensively studied and
used to provide series of examples with interesting dynamical
properties (see for example \cite{Do1,Do2,GJ,Wi}). Generalizations
of Toeplitz subshifts and some of their properties to more general
group actions  can be found in \cite{Co06,CP,Do3,Kr,Kr2}. For
instance, in \cite{CP} Toeplitz subshifts are characterized as the
minimal symbolic almost 1-1 extensions of odometers (see \cite{DL}
for this result in the context of $\ZZ$-actions). In this paper,
we give an explicit construction that generalizes the result of
Downarowicz in \cite{Do1}, to Toeplitz subshifts given by actions
of groups which are amenable, countable and  residually finite. The following is our main result.

\begin{mainA}\label{main-theorem}
Let $G$ be an infinite, countable,  amenable and residually finite  
group. For every metrizable Choquet simplex $K$ and any
$G$-odometer $O$, there exists a Toeplitz $G$-subshift whose set
of invariant probability measures is affine homeomorphic to $K$
and such that it is an almost 1-1 extension of $O$.
\end{mainA}
Typical examples of the groups $G$ involved in  this theorem  are the finitely generated
  subgroups of upper triangular matrices in
$GL(n,\CC)$.

The strategy of  Downarowicz in \cite{Do1}, is to construct an affine homeomorphism between an arbitrary 
metrizable Choquet simplex $K$ and a subset of the space of invariant probability measures of the full shift $\{0,1\}^{\ZZ}$. 
Then he shows  it coincides with the space of invariant probability measures of a Toeplitz subshift $Y\subseteq \{0,1\}^{\ZZ}$. 
To do this, he uses the structure of metric space of the space of measures.
In this paper we consider the representation of $K$ as an inverse limit  of finite dimensional simplices with linear transition maps $(M_n)_n$.
Then we use this transition maps to construct Toeplitz $G$-subshifts having sequences of Kakutani-Rokhlin partitions with $(M_n)_{n}$ 
as the associated sequence of incidence matrices. Our approach is closer to the strategy used in \cite{GJ} by Gjerde and Johansen, and deals 
with the combinatorics of F\o lner sequences.

We obtain, furthermore  some consequences in the orbit
equivalence problem. Two minimal Cantor systems  are
(topologically) orbit equivalent, if there exists an
orbit-preserving  homeomorphism between their phase spaces.
Giordano, Matui, Putnam and Skau show in \cite{GMPS} that every
minimal $\ZZ^d$-action on the Cantor set is orbit equivalent to a
minimal $\ZZ$-action.  It is still unknown  if  every minimal action of a countable amenable group
on the Cantor set is orbit equivalent to a $\ZZ$-action.  Nevertheless it is clear that the result in   \cite{GMPS}
  can not be extended to any
countable group. For instance, by using the notion of cost,
Gaboriau \cite{Ga} proves that if two free actions of free groups
$\FF_{n}$ and $\FF_{p}$ are  (even measurably) orbit equivalent
then their rank are the same i.e. $n=p$.  Another problem is to know  which are the
$\ZZ$-orbit equivalence classes that the $\ZZ^d$-actions (or more
general group actions) realize. We give a partial answer for 
this question. As a consequence of the proof of Theorem A we
obtain the following result.

\begin{mainB}\label{main-corollary}
Let $(X,\sigma|_X,\ZZ)$ be a Toeplitz $\ZZ$-subshift. Then  for every $d\geq 1$ there exists a Toeplitz $\ZZ^d$-subshift
which is orbit equivalent to $(X,\sigma|_X,\ZZ)$.  
\end{mainB}

This paper is organized as follows. Section 2 is devoted to 
introduce the basic definitions. For an amenable discrete group $G$ and a decreasing
sequence of finite index subgroups of $G$ with trivial
intersection, we construct in Section 3 an associated sequence
$(F_n)_{n\geq 0}$ of fundamental domains,  so that it is F\o lner
and each $F_{n+1}$ is tilable by translated copies of $F_n$.  In Section 4 we construct Kakutani-Rokhlin partitions for generalized Toeplitz subshifts, and in Section 5 we use the fundamental domains introduced in Section 3 to construct Toeplitz subshifts having sequences of Kakutani-Rokhlin partitions with a prescribed sequence of incidence matrices. This construction improves and generalizes that one
given in \cite{Co} for $\ZZ^d$-actions, and moreover, allows to
characterize the associated ordered group with unit. In Section 6   we give a
characterization of any Choquet simplex as an inverse limit
defined by    sequences of matrices that we use in Section 5 (they are called "managed" sequences).  Finally, in Section 7 we use the previous results to prove Theorems A and B.

\section{Basic definitions and background}
In this article,  by a  {\em topological dynamical system} we mean
a triple  $(X, T, G)$, where $T$ is a continuous left action of
a  countable  group $G$ on the
compact metric space $(X,d)$. For every $g\in G$, we denote $T^g$
the homeomorphism that induces the action of $g$ on $X$. The unit
element of $G$ will be called $e$. The system $(X, T, G)$ or the
action $T$ is {\it minimal} if for every $x\in X$  the orbit
$o_T(x)=\{T^g(x): g\in G\}$ is dense in $X$. We say that $(X,T,G)$
is a {\it minimal Cantor system} or a minimal Cantor $G$-system if
$(X,T,G)$ is a minimal topological dynamical system with $X$ a
Cantor set.

An {\it invariant probability measure} of the topological
dynamical system $(X, T, G)$ is a probability Borel measure $\mu$
such that $\mu(T^g(A))=\mu(A)$, for every Borel set $A$. We denote
by $\M(X, T, G)$ the space of invariant probability measures of
$(X, T, G)$.


\subsection{Subshifts.}

For every $g\in G$, denote $L_g:G\to G$ the left multiplication by
$g\in G$. That is, $L_g(h)=gh$ for every $h\in G$. Let $\Sigma$ be
a finite alphabet.  $\Sigma^G$ denotes the set of all the
functions $x:G\to \Sigma$. The (left) {\it shift action} $\sigma$
of $G$ on $\Sigma^{G}$ is given by $\sigma^g(x)=x\circ
L_{g^{-1}}$, for every $g\in G$. Thus $\sigma^g(x)(h)=x(g^{-1}h)$.
We consider $\Sigma$ endowed with the discrete topology and
$\Sigma^{G}$ with the product topology. Thus every $\sigma^g$ is a
homeomorphism of the Cantor set $\Sigma^G$. The topological
dynamical system $(\Sigma^G,\sigma,G)$ is called the full
$G$-shift on $\Sigma$. For every finite subset $D$ of $G$ and $x\in
\Sigma^{G}$, we denote $x|_D\in\Sigma^{D}$ the restriction of $x$
to $D$. For $F\in \Sigma^D$ ($F$ is a function from $D$ to
$\Sigma$) we denote by $[F]$ the set of all $x\in \Sigma^D$ such
that $x|_D=F$. The set $[F]$ is called the {\it cylinder} defined
by $F$, and it is a clopen set (both open and closed). The
collection of all the sets $[F]$ is a base of the topology of
$\Sigma^{G}$. 

\begin{definition}A {\it subshift } or $G$-subshift of $\Sigma^{G}$ is a closed subset
$X$ of $\Sigma^{G}$ which is invariant by the shift action.
\end{definition}
 The
topological dynamical system $(X,\sigma|_X, G)$ is also called
subshift or $G$-subshift. See \cite{CC} for details.

\subsubsection{Toeplitz $G$-subshifts.} An element $x\in \Sigma^G$
is a {\it Toeplitz sequence}, if for every $g\in G$ there exists a
finite index subgroup $\Gamma$ of $G$ such that
$\sigma^{\gamma}(x)(g)=x(\gamma^{-1}g)=x(g)$, for every
$\gamma\in\Gamma$.

A subshift $X\subseteq \Sigma^{G}$ is a {\it Toeplitz subshift} or
Toeplitz $G$-subshift if there exists a Toeplitz sequence $x\in
\Sigma^{G}$ such that $X=\overline{o_{\sigma}(x)}$. 
It is shown in \cite{CP}, \cite{Kr} and \cite{Kr2} that a Toeplitz sequence $x$ is
{\it regularly recurrent}, {\em i.e.} for every neighborhood $V$ of $x$ there exists a finite index subgroup $\Gamma$ of $G$ such that $\sigma^{\gamma}(x)\in V$, for every $\gamma\in\Gamma$. This condition is stronger than almost periodicity, which implies minimality of the closure of the orbit of $x$ (see \cite{Aus} for details about almost periodicity).

\subsection{Inverse and direct limit.} Given a sequence of continuous maps $f_{n} \colon X_{n+1} \to X_{n},  n\ge 0$ on topological spaces $X_{n}$, we denote  the  associated {\em inverse limit} by
\begin{eqnarray*}
 \lim_{\leftarrow n} (X_{n},f_{n}) &= &\xymatrix{
{X_{0}}\ar@{<-}[r]^{f_{0}} &
{X_{1}}\ar@{<-}[r]^{f_{1}} &
{X_{2}}\ar@{<-}[r]^{f_{2}}& {\cdots}}\\
\ & := & \{ (x_{n})_{n} ; x_{n } \in X_{n}, \ x_{n}=f_{n}(x_{n+1})   \ \forall n \ge0\}.
 \end{eqnarray*}
Let us recall that this space is compact  when all the spaces $X_{n}$ are compact and  the inverse limit spaces associated to any increasing  subsequences $(n_{i})_{i}$ of indices  are homeomorphic.

In a similar way, we denote for a sequence of maps $g_{n} \colon X_{n} \to X_{n+1},  n\ge 0$ the  associated {\em direct limit} by
\begin{eqnarray*}
\lim_{\rightarrow n} (X_{n},g_{n})  &= &\xymatrix{
{X_{0}}\ar@{->}[r]^{g_0} & {X_{1}}\ar@{->}[r]^{g_1} & {X_{2}}\ar@{->}[r]^{g_2}& {\cdots}}\\
\ &:= &  \{ (x,n), x \in X_{n} , \ n\ge 0 \}/\sim,
\end{eqnarray*}

where  two elements are equivalent $ (x,n) \sim(y,m)$ if and only
if there exists $k \ge m,n$ such that $ g_{k}\circ \ldots \circ
g_{n}(x) = g_{k}\circ \ldots \circ g_{m}(x)$. We denote by $[x,n]$
the equivalence class of $(x,n)$. When the maps $g_{n}$ are
homomorphisms on groups $X_{n}$, then the direct limit inherits a group
structure.

\subsection{Odometers.}
A group $G$ is said to be {\em residually finite} if there exists a
nested sequence $(\Gamma_n)_{n \ge 0}$  of finite index normal
subgroups such that $\bigcap_{n \ge 0} \Gamma_n$ is trivial.  For every $n\ge 0$, there
exists then a canonical projection $\pi_n \colon
G/\Gamma_{n+1} \to G/\Gamma_{n}$. The  $G$-{\em odometer} or {\em
adding machine} $O$ associated to the sequence $(\Gamma_n)_n$ is
the inverse limit
\begin{eqnarray*}
  O := \lim_{\leftarrow n} (G/\Gamma_{n},\pi_{n}) &= &\xymatrix{
{G/\Gamma_{0}}\ar@{<-}[r]^{\pi_{0}} &
{G/\Gamma_{1}}\ar@{<-}[r]^{\pi_{1}} &
{G/\Gamma_{2}}\ar@{<-}[r]^{\pi_{2}}& {\cdots}}.
\end{eqnarray*}
We refer to \cite{CP} for the  basic properties of such a space.
Let us recall that it inherits a group structure through the
quotient groups $G/\Gamma_{n}$ and it contains $G$ as a subgroup
thanks the injection $G \ni g \mapsto ([g]_n) \in O$, where
$[g]_n$ denotes the class of $g $ in $G/\Gamma_{n}$. Thus the
group $G$ acts by left multiplication on $O$. When there is no
confusion, we call  this action also odometer. It is
equicontiuous, minimal and the  left Haar measure is the unique
invariant probability measure. Notice that this action is free:
the stabilizer of any point is trivial. The Toeplitz $G$-subshifts
are characterized as the subshifts that are minimal almost 1-1
extensions of $G$-odometers \cite{CP}.

\subsection{Ordered groups} 
For more details about ordered groups and dimension groups we
refer to \cite{E} and \cite{Go}.

An {\it ordered group} is a pair
$(H,H^+)$, such that $H$ is a countable abelian group and $H^+$ is
a subset of $H$ verifying $(H^+)+(H^+)\subseteq H^+$,
$(H^+)+(-H^+)=H$ and $(H^+)\cap (-H^+)=\{0\}$ (we use $0$ as the
unit of $H$ when $H$ is abelian). An  ordered group $(H,H^+)$ is a
{\it dimension group} if for every $n\in\NN$ there exist $k_n\geq
1$ and a positive homomorphism $A_n:\ZZ^{k_n}\to \ZZ^{k_{n+1}}$,
such that $(H,H^+)$ is isomorphic to $(J,J^+)$, where $J$ is the
direct limit
$$
\lim_{\longrightarrow n}(\ZZ^{k_n},A_n)=\xymatrix{
{\ZZ^{k_0}}\ar@{->}[r]^{A_0} & {\ZZ^{k_1}}\ar@{->}[r]^{A_1} &
{\ZZ^{k_2}}\ar@{->}[r]^{A_2}& {\cdots}},
$$
and $J^+=\{[v,n]: a\in (\ZZ^+)^{k_n}, n\in\NN\}$. The dimension group is
{\it simple} if the matrices $A_n$ can be chosen strictly
positive.

An {\it order  unit} in the ordered group $(H,H^+)$ is an element
$u\in H^+$ such that for every $g\in H$ there exists $n\in \NN$
such that $nu-g \in H^+$. If $(H,H^+)$ is a simple dimension group
then each element in $H^+\setminus\{0\}$ is an order unit. A {\it
unital ordered group} is a triple $(H,H^+,u)$ such that $(H,H^+)$
is an ordered group and $u$ is an order unit. An isomorphism
between two unital ordered groups $(H,H^+,u)$ and $(J,J^+,v)$ is
an isomorphism  $\phi:H\to J$ such that $\phi(H^+)=J^+$ and
$\phi(u)=v$. A {\it state} of the unital ordered group $(H,H^+,
u)$ is a homomorphism $\phi: H\to \RR$ so that $\phi(u)=1$ and
$\phi(H^+)\subseteq \RR^+$. The {\it infinitesimal subgroup} of a simple dimension group with unit
$(H,H^+,u)$ is
$$
\inf(H)=\{a\in H: \phi(a)=0 \mbox{ for all state } \phi \}.
$$
It is not difficult to show that $\inf(H)$ does not depend on the
order unit.

The quotient group $H/\inf(H)$ of a simple dimension group
$(H,H^+)$ is also a simple dimension group with positive cone
$$(H/\inf(H))^+=\{[a]: a\in H^+\}.$$

The next result  is well-known.  The proof is left to the reader.

\begin{lemma}\label{nexo}
Let $(H,H^+)$ be a simple dimension group equals  to
the direct limit
$$
\lim_{\rightarrow n}(\ZZ^{k_n},M_n)=\xymatrix{
{\ZZ^{k_0}}\ar@{->}[r]^{M_0} & {\ZZ^{k_1}}\ar@{->}[r]^{M_1} &
{\ZZ^{k_2}}\ar@{->}[r]^{M_2}& {\cdots}}.
$$
Then for every $z=({z}_n)_{n \geq 0}$ in the inverse limit
$$
\lim_{\leftarrow n}((\RR^+)^{k_n},M_n^T)=\xymatrix{
{(\RR^+)^{k_0}}\ar@{<-}[r]^{M_0^T} &
{(\RR^+)^{k_1}}\ar@{<-}[r]^{M_1^T} &
{(\RR^+)^{k_2}}\ar@{<-}[r]^{M_2^T}& {\cdots}},
$$
the function $\phi_z: H\to \RR$  given by
$\phi([n,{v}])=<{v},{z}_n>$,  for every
$[n,{v}]\in H,$ is well defined and  is a homomorphism of
groups such that $\phi_z(H^+)\subseteq \RR^+$. Conversely, for
every group homomorphism $\phi: H\to \RR$ such that
$\phi(H^+)\subseteq \RR^+$, there exists a unique $z\in
\lim_{\leftarrow n}((\RR^+)^{k_n},M_n^T)$ such that $\phi=\phi_z$.
\end{lemma}

The following lemma is a preparatory lemma to prove  Theorem A and B. 

\begin{lemma}\label{nexo2}
Let $(H,H^+,u)$ be a simple dimension group with unit given by the following  
 direct limit
$$
\lim_{\rightarrow n}(\ZZ^{k_n},A_n)=\xymatrix{
{\ZZ^{}}\ar@{->}[r]^{A_0} & {\ZZ^{k_1}}\ar@{->}[r]^{A_1} &
{\ZZ^{k_2}}\ar@{->}[r]^{A_2}& {\cdots}},
$$
with unit $u=[1,0]$.  Suppose that  $A_n> 0$ for every $n\geq 0$. 
Then $(H,H^+, u)$ is isomorphic to  
$$
\xymatrix{
{\ZZ^{}}\ar@{->}[r]^{\tilde{A}_0} & {\ZZ^{k_1+1}}\ar@{->}[r]^{\tilde{A}_1} &
{\ZZ^{k_2+1}}\ar@{->}[r]^{\tilde{A}_2}& {\cdots}},
$$
where  $\tilde{A}_0$ is the $(k_1+1)\times 1$-dimensional matrix given by
$$
\tilde{A}_0=\left( \begin{array}{c}
                               A_0(1,\cdot)\\
                              A_0(1,\cdot)\\
                              A_0(2, \cdot)\\
                              \vdots\\
                              A_0(k_1,\cdot)
                                  \end{array}
  \right),
$$
and $\tilde{A}_n$ is the $(k_{n+1}+1)\times (k_n+1)$ dimensional matix given by
$$
\tilde{A}_n=\left(\begin{array}{ccccc }
                     1 & A_n(1,1)-1 & A_n(1,2) & \cdots & A_n(1,k_n)\\
                     1 & A_n(1,1)-1 & A_n(1,2) & \cdots & A_n(1,k_n)\\
                     1 & A_n(2,1)-1 & A_n(2,2) & \cdots & A_n(2,k_n)\\
                     \vdots & \vdots & \vdots &  & \vdots\\
                   1& A_n(k_{n+1},1)-1 & A_n(k_{n+1},2) & \cdots & A_n(k_{n+1}, k_n)\\                      
\end{array}
                      \right), \mbox{ for every } n\geq 0.
$$
 \end{lemma}

\begin{proof}
For $n\geq 1$,  consider $M_n$ the $(k_n+1)\times k_n$-dimensional matrix given by
$$
M_n(\cdot,k)=\left\{ \begin{array}{lll}
                                 \vec{e}_{n,1}+\vec{e}_{n,2} & \mbox{ if } & k=1\\
                                               \vec{e}_{k+1} & \mbox{ if } & 3\leq k \leq k_n\\     
                                   \end{array}\right.,
$$
where $\vec{e}_{n,1}, \cdots, \vec{e}_{n,k_n+1}$ are the canonical vectors in $\RR^{k_n+1}$.
Let $B_n$ be the $k_{n+1}\times (k_n+1)$-dimensional matrix defined by
$$
B_n(i,j)=\left\{ \begin{array}{lll}
                         1 &\mbox{ if } &  j=1\\
                           A_n(i,1)-1 &  \mbox{ if } & j=2\\
                          A_n(i,j-1) &\mbox{ if } & 3\leq j\leq k_n+1\\
\end{array}
\right.
$$
We have  $A_n=B_nM_n$ and $\tilde{A}_n=M_{n+1}B_n$  for every $n\geq 1$, and $\tilde{A}_0=M_1A_0$.   

Thus the Bratteli diagrams defined by the sequences of matrices $(A_n)_{n\geq 0 }$ 
 and  $(\tilde{A}_n)_{n\geq 0}$  are contractions  of the same diagram. This shows that the respective dimension groups with unit are isomorphic (see \cite{GPS} or \cite{Du}).
 \end{proof}

\subsection{Associated ordered group and orbit equivalence.}
Let $(X,T,G)$ be a topological dynamical system such that $X$ is a
Cantor set and $T$ is minimal. The ordered  group associated to
$(X,T,G)$ is the unital ordered group $$\G(X,T,
G)=(D_m(X,T,G),D_m(X,T,G)^+,[1]),$$ where
$$
D_m(X, T,G)= C(X,\ZZ)/\{f\in C(X,\ZZ): \int fd\mu =0, \forall
\mu\in \M(X,T,G)\},
$$
$$
D_m(X,T,G)^+=\{[f]: f\geq 0\},
$$
and $[1]\in D_m(X,T,G)$ is the class of the constant function $1$.
\medskip

Two topological dynamical systems $(X_1,T_1,G_1)$ and
$(X_2,T_2,G_2)$ are (topologically) {\it orbit equivalent} if there
exists a homeomorphism $F:X_1\to X_2$ such that
$F(o_{T_1}(x))=o_{T_2}(F(x))$ for every $x\in X_1$.

In \cite{GMPS} the authors show the following algebraic caracterization of orbit equivalence.
\begin{theorem}[\cite{GMPS}, Theorem 2.5] \label{theorem1}Let  $(X,T, \ZZ^d) $ and $(X', T', \ZZ^m)$ be two minimal actions  on the Cantor set. Then they are orbit equivalent if and only if 
$$ \G(X,T, \ZZ^d) \simeq \G(X',T', \ZZ^m)$$
as isomorphism of unital ordered group. 
\end{theorem}

\section{Suitable F\o lner sequences.}\label{suitable-sequence}

Let $G$ be a  residually finite group, and  let $(\Gamma_n)_{n\geq 0}$
be a  nested sequence of finite index normal subgroup of $G$ such that
$\bigcap_{n\geq 0}\Gamma_n=\{e\}$.

For technical reasons it is important to notice that since the groups $\Gamma_n$ are normal,
we have $g\Gamma_n=\Gamma_ng$,  for every  $g\in G$.

To construct a Toeplitz $G$-subshift  that is an almost 1-1 extension of the odometer defined by the sequence $(\Gamma_n)_n$,
 we need a ``suitable'' sequence $(F_n)_n$ of fundamental domains of $G/\Gamma_n$. More precisely, each $F_{
n+1} $ has to be tileable by translated copies of $F_n$.
To control the simplex of invariant measures of the subshift, we need in addition the sequence $(F_n)_n$ to be F\o lner.
We did not find in the specialized litterature  a result  ensuring these conditions.

\subsection{Suitable sequence of fundamental domains.}

Let $\Gamma$ be a normal subgroup of $G$. By a {\it fundamental domain} of $G/\Gamma$, we mean a subset $D\subseteq G$ containing exactly one representative element of each equivalence class in $G/\Gamma$.

\begin{lemma} \label{folner0}
Let $(D_n)_{n\geq 0}$ be an increasing sequence of finite subsets
of $G$  such that for every $n\geq 0$, $e\in D_n$ and $D_n$ is a
fundamental domain of $G/\Gamma_n$. Let $(n_i)_{i\geq 0}\subseteq
\NN$ be an increasing sequence. Consider  $(F_i)_{i\geq 0}$
defined by $F_0=D_{n_0}$ and
$$F_{i}=\bigcup_{v\in D_{n_i}\cap \Gamma_{n_{i-1}}}vF_{i-1} \mbox{ for every } i\geq 1.$$
Then for every $i\geq 0$ we have the following:
\begin{enumerate}
\item $F_i\subseteq F_{i+1}$ and $F_i$ is a fundamental domain of
$G/\Gamma_{n_i}$.

\item $F_{i+1}=\bigcup_{v\in F_{i+1}\cap \Gamma_{n_i}} vF_i$.
\end{enumerate}
\end{lemma}

\begin{proof}
Since $e\in D_{n_i}$, the sequence  $(F_{i})_{i\geq 0}$ is
increasing.

$F_0=D_{n_0}$ is a fundamental domain of $G/\Gamma_{n_0}$. We will
prove by induction on $i$  that $F_i$ is a fundamental domain of
$G/\Gamma_{n_i}$. Let $i>0$ and suppose that $F_{i-1}$ is a
fundamental domain of $G/\Gamma_{n_{i-1}}$.

Let $v\in D_{n_i}$.  There exist  then $u\in F_{i-1}$ and $w\in
\Gamma_{n_{i-1}}$ such that $v=wu$. Let $z\in D_{n_i}$ and
$\gamma\in \Gamma_{n_i}$ be such that $w=\gamma z$. Since $z\in
\Gamma_{n_{i-1}}\cap D_{n_i}$ and $v= \gamma zu$, we conclude
that $F_i$ contains one representing element of each class in
$G/\Gamma_{n_i}$.

Let $w_1,w_2\in F_i$ be such that there exists $\gamma \in
\Gamma_{n_i}$ verifying $w_1=\gamma w_2$. By definition,
$w_1=v_1u_1$ and $w_2=v_2u_2$, for some $u_1,u_2\in F_{i-1}$ and
$v_1,v_2\in D_{n_i}\cap\Gamma_{n_{i-1}}$. This implies that $u_1$
and $u_2$ are in the same class of $G/\Gamma_{n_{i-1}}$. Since
$F_{i-1}$ is a fundamental domain, we have $u_1=u_2$. From this we
get $v_1=\gamma v_2$, which implies that $v_1=v_2$. Thus we deduce
that $F_i$ contains at most one representing element of each class
in $G/\Gamma_{n_i}$. This shows that $F_i$ is a fundamental domain
of $G/\Gamma_{n_i}$.

To show that $D_{n_i}\cap \Gamma_{n_{i-1}}\subseteq F_i\cap \Gamma_{n_{i-1}}$, observe that the definition of $F_i$ implies that for every $v\in  D_{n_i}\cap \Gamma_{n_{i-1}}$ and $u\in F_{i-1}$, $vu\in F_i$. Then for $u=e\in F_{i-1}$ we get $v=ve\in F_i$. Now suppose that $v\in F_i\cap \Gamma_{n_{i-1}}\subseteq F_i$. The definition of $F_i$ implies there exist $u\in F_{i-1}$ and $\gamma\in D_{n_i}\cap \Gamma_{n_{i-1}}$ such that $v=\gamma u$.  Since $v$ and $\gamma$ are in $\Gamma_{n_{i-1}}$,  we get that $u\in \Gamma_{n_{i-1}}\cap F_{i-1}$. This implies that $u=e$ because  $\Gamma_{n_{i-1}}\cap F_{i-1}=\{e\}$.
\end{proof}

{\rm In this paper, by F\o lner sequences we mean right F\o lner
sequences. That is, a sequence  $(F_n)_{n\geq 0}$ of nonempty
finite sets of $G$  is a F\o lner sequence   if   for every $g\in
G$ $$ \lim_{n\to \infty} \frac{|F_ng\triangle F_n|}{|F_n|}=0.
$$
Observe that $(F_n)_{n\geq 0}$ is a right F\o lner sequence if
and only if $(F_n^{-1})_{n\geq 0}$ is a left F\o lner sequence. }

\begin{lemma}\label{folner}
Suppose that $G$ is amenable. There exists an increasing sequence
$(n_i)_{i\geq 0}\subseteq \NN$ and a F\o lner sequence
$(F_{i})_{i\in \NN}$, such that
\begin{enumerate}[i)]

\item $F_i\subseteq F_{i+1}$ and  $F_{i}$ is a fundamental domain
of $G/\Gamma_{n_i}$, for every $i\geq 0$.

\item $G=\bigcup_{i\geq 0}F_{i}.$

\item $F_{i+1}=\bigcup_{v\in F_{i+1}\cap \Gamma_{n_i}} vF_i$, for
every $i\geq 0$.
\end{enumerate}
\end{lemma}
\begin{proof}
From \cite[Theorem 1]{Wei} (see  \cite[Proposition 4.1]{Kr} for a proof in our context), there exists an increasing
sequence $(m_i)_{i\geq 0}\subseteq \NN$ and a F\o lner sequence
$(D_{i})_{i\in \NN}$ such that for every $i\geq 0$, $D_i\subseteq
D_{i+1}$,  $D_{i}$ is a fundamental domain of $G/\Gamma_{m_i}$,
and  $G=\bigcup_{i\geq 0}D_{i}$. Up to take subsequences, we
can assume that $D_i$ is a fundamental domain of $G/\Gamma_i$, for
every $i\geq 0$, and that $e\in D_0$.

We will construct the sequences $(n_i)_{i\geq 0}$ and
$(F_n)_{n\geq 0}$ as follows:

\medskip

{\it Step $0$:} We set $n_0=0$ and $F_0=D_0$.

\medskip

{\it Step $i$:}  Let $i>0$. We assume that we have chosen $n_j$
and $F_j$ for every $0\leq j<i$. We take $n_i>n_{i-1}$ in order
that the following two conditions are verified:
\begin{equation}\label{folner-eq1}
\frac{|D_{n_i}g\vartriangle
D_{n_i}|}{|D_{n_i}|}<\frac{1}{i|F_{i-1}|}, \mbox{ for every } g\in
F_{i-1}.
\end{equation}

\begin{equation}\label{folner-eq2}
D_{n_{i-1}}\subseteq \bigcup_{v\in D_{n_i}\cap
\Gamma_{n_{i-1}}}vF_{i-1}.
\end{equation}

Such integer $n_i$ exists  because  $(D_n)_{n\geq 0}$ is a F\o lner
sequence and $F_{i-1}$ is a fundamental domain of
$G/\Gamma_{n_{i-1}}$ (then $G=\bigcup_{v\in
\Gamma_{n_{i-1}}}vF_{i-1}$).

We define $$F_{i}=\bigcup_{v\in D_{n_i}\cap
\Gamma_{n_{i-1}}}vF_{i-1}.$$

Lemma \ref{folner0} ensures that $(F_i)_{i\geq 0}$ verifies {\em i)}
and  {\em iii)} of the lemma.  The equation (\ref{folner-eq2}) implies
that $(F_i)_{i\geq 0}$ verifies  {\em ii)} of the lemma.

It remains to show that $(F_i)_{i\geq 0}$ is a F\o lner sequence.

By definition of $F_i$ we have
$$
\left( F_i\setminus D_{n_i}\right)\subseteq \bigcup_{g\in F_{i-1}}
\left(D_{n_i}g\setminus D_{n_i}\right).
$$
Then by equation (\ref{folner-eq1}) we get
\begin{eqnarray*}
\frac{|F_i\setminus D_{n_i}|}{|D_{n_i}|}& \leq & \sum_{g\in
F_{i-1}}\left(\frac{|D_{n_i}g\setminus D_{n_i}|}{|D_{n_i}|}\right)\\
    &\leq &\left( |F_{i-1}|\frac{1}{i|F_{i-1}|}\right)=\frac{1}{i}.
\end{eqnarray*}

Since $$\left(|F_i\cap D_{n_i}|+|D_{n_i}\setminus
F_i|\right)=|D_{n_i}|=|F_i|=|F_i\cap D_{n_i}|+|F_i\setminus D_{n_i}|,$$
we obtain
$$
\frac{|D_{n_i}\setminus F_i|}{|D_{n_i}|}\leq \frac{1}{i}.
$$

Let $g\in G$. Since
\begin{eqnarray*}
F_i g\setminus F_i & = & \left[ (F_i\cap D_{n_i})g\setminus F_i\right] \bigcup \left[ (F_i\setminus D_{n_i})g\setminus F_i \right] \\
      &\subseteq & \left[ (F_i\cap D_{n_i})g\setminus F_i \right] \bigcup (F_i\setminus D_{n_i})g \\
      &\subseteq & \left[D_{n_i}g\setminus (F_i\cap D_{n_i})  \right] \bigcup (F_i\setminus D_{n_i})g, \\
\end{eqnarray*}
we have
\begin{equation}\label{folner-eq3}
\frac{|F_ig\setminus F_i|}{|F_i|}\leq \frac{|D_{n_i}g\setminus
(F_i\cap D_{n_i})|}{|D_{n_i}|}+\frac{|(F_i\setminus
D_{n_i})g|}{|D_{n_i}|}\leq \frac{|D_{n_i}g \setminus (F_i\cap
D_{n_i})|}{|D_{n_i}|}+\frac{1}{i}.
\end{equation}

On the other hand, the relation
$$D_{n_i}g\setminus D_{n_i}=D_{n_i}g\setminus \left[ (D_{n_i}\cap
F_i)\cup (D_{n_i}\setminus F_i )\right] = \left[D_{n_i}g\setminus(D_{n_i}\cap
F_i)\right] \setminus (D_{n_i}\setminus F_i),
$$
implies that
\begin{eqnarray*}
D_{n_i}g \setminus (F_i\cap D_{n_i})&=& [(D_{n_i}g \setminus
(F_i\cap D_{n_i}))\cap(D_{n_i}\setminus F_i)]\bigcup \left[ (D_{n_i}g
\setminus
(F_i\cap D_{n_i}))\setminus (D_{n_i}\setminus F_i) \right] \\
&=& [(D_{n_i}g \setminus (F_i\cap D_{n_i}))\cap(D_{n_i}\setminus
F_i)]\bigcup \left[ D_{n_i}g\setminus D_{n_i}\right] \\
& \subseteq & (D_{n_i}\setminus F_i) \bigcup  (D_{n_i}g\setminus D_{n_i}),
\end{eqnarray*}
which ensures that
\begin{equation}\label{folner-eq4}
\frac{|D_{n_i}g \setminus (F_i\cap D_{n_i}) |}{|D_{n_i} |}\leq
\frac{|D_{n_i}\setminus F_i|}{|D_{n_i}|}+\frac{|D_{n_i}g\setminus
D_{n_i}|}{|D_{n_i}|}.
\end{equation}
From equations (\ref{folner-eq3}) and (\ref{folner-eq4}), we
obtain
$$
\frac{|F_i g\setminus F_i|}{|F_i|}\leq
\frac{2}{i}+\frac{|D_{n_i}g\setminus D_{n_i}|}{|D_{n_i}|},
$$
which implies
\begin{equation}\label{folner-eq5}
\lim_{i\to\infty}\frac{|F_ig\setminus F_i|}{|F_i|}=0.
\end{equation}
In a similar way we deduce that
$$
F_i\setminus F_ig\subseteq [D_{n_i}\setminus (F_i\cap
D_{n_i})g]\bigcup (F_i\setminus D_{n_i}),
$$
$$
D_{n_i}\setminus D_{n_i}g=[D_{n_i}\setminus (D_{n_i}\cap
F_i)g] \setminus (D_{n_i}\setminus F_i),
$$
and
$$
D_{n_i}\setminus (F_i\cap D_{n_i})g\subseteq (D_{n_i}\setminus
F_i) \bigcup (D_{n_i}\setminus D_{n_i}g).
$$
Combining the last three equations we get
$$
\frac{|F_i\setminus F_i g|}{|F_i|}\leq
\frac{2}{i}+\frac{|D_{n_i}\setminus D_{n_i}g|}{|D_{n_i}|},
$$
which implies
\begin{equation}\label{folner-eq6}
\lim_{i\to\infty}\frac{|F_i\setminus F_ig|}{|F_i|}=0.
\end{equation}
Equations (\ref{folner-eq5}) and (\ref{folner-eq6}) imply that
$(F_i)_{i\geq 0}$ is F\o lner.
\end{proof}

The following result is a direct consequence of Lemma
\ref{folner}.

\begin{lemma}\label{folner2}
Let $G$ be an amenable  residually finite group and let  $(\Gamma_n)_{n\geq 0}$ be
a decreasing sequence of finite index normal subgroups of $G$ such that $\bigcap_{n\geq
0}\Gamma_n=\{e\}$. There exists an increasing sequence
$(n_i)_{i\geq 0}\subseteq \NN$ and a F\o lner sequence $(F_i)_{i\geq 0}$ of
$G$ such that
\begin{enumerate}
\item $\{e\}\subseteq F_i\subseteq F_{i+1}$ and  $F_{i}$ is a
fundamental domain of $G/\Gamma_{n_i}$, for every $i\geq 0$.

\item $G=\bigcup_{i\geq 0}F_{i}.$

\item $F_{j}=\bigcup_{v\in F_{j}\cap \Gamma_{n_i}} vF_i$, for every
$j>i\geq 0$.
\end{enumerate}
\end{lemma}

\begin{proof}
The existence of the sequence of subgroups of $G$ and the F\o lner
sequence verifying (1), (2) and (3) for $j=i+1$ is direct from
Lemma \ref{folner}. Using induction, it is straightforward to show
(3) for every $j>i\geq 0$.
\end{proof}

\section{Kakutani-Rokhlin partitions for generalized Toeplitz subshifts} \label{K-R partitions} 

In this section $G$ is an amenable, countable,  and residually finite group. 

Let $\Sigma$ be a finite alphabet and let $(\Sigma^G,\sigma, G)$ be the respective full $G$-shift.

For a finite index subgroup $\Gamma$ of $G$, $x\in\Sigma^G$ and $a\in \Sigma$, we define
$$
Per(x,\Gamma,a)=\{g\in G: \sigma^{\gamma}(x)(g)=x(\gamma^{-1}g)=a, \forall \gamma\in \Gamma \},
$$
and  $Per(x,\Gamma)=\bigcup_{a\in\Sigma} Per(x,\Gamma,a).$

It is straightforward to show that  $x\in\Sigma^G$ is a Toeplitz sequence if and only if there exists an increasing  sequence
$(\Gamma_n)_{n\geq 0}$ of finite index subgroups of $G$ such that $G=\bigcup_{n\geq 0}Per(x,\Gamma_n)$  (see  \cite[Proposition 5]{CP}).

A {\it period structure} of $x\in\Sigma^{G}$ is an increasing sequence of finite index subgroups $(\Gamma_n)_{n\geq 0}$ of $G$ such that 
$G=\bigcup_{n\geq 0}Per(x,\Gamma_n)$ and such that for every $n\geq 0$, $\Gamma_n$ is an {\it essential group of periods}: 
This  means   that if $g\in G$ is such that $Per(x,\Gamma_n,a)\subseteq Per(\sigma^g(x),\Gamma_n,a)$ for every $a\in\Sigma$, 
then $g\in\Gamma_n$.

It is known that every Toeplitz sequence has a period structure (see for example \cite[Corollary 6]{CP}). We construct in this section,
thanks the period structure, a Kakutani-Rokhlin partition  and we deduce a characterization of its ordered group.

\subsection{Existence of Kakutani-Rokhlin partitions.}
In this subsection we suppose that $x_0\in \Sigma^{G}$ is a non-periodic Toeplitz sequence ($\sigma^g(x_0)=x_0$ implies $g=e$) 
having a period structure $(\Gamma_n)_{n\geq 0}$ such that for every $n\geq 0$,
\begin{itemize}
\item[(i)] $\Gamma_{n+1}$ is a proper subset of $\Gamma_{n}$,

\item[(ii)] $\Gamma_n$ is a normal subgroup of $G$.
\end{itemize}

Every non-periodic Toeplitz sequence has a period structure verifying (i) \cite[Corollary 6]{CP}.
Condition (ii) is satisfied for every Toeplitz sequence whose  Toeplitz subshift is an almost 1-1 extension of an odometer 
(in the general case these systems are almost 1-1 extensions of subodometers. See \cite{CP} for the details).

 By Lemma \ref{folner2} we can   assume there exists  
     a F\o lner sequence $(F_n)_{n\geq 0}$ of $G$ such that
\begin{itemize}
\item[(F1)] $\{e\}\subseteq F_n\subseteq F_{n+1}$ and  $F_{n}$ is a
fundamental domain of $G/\Gamma_{n}$, for every $n\geq 0$.

\item[(F2)] $G=\bigcup_{n\geq 0}F_{n}.$

\item[(F3)] $F_{n}=\bigcup_{v\in F_{n}\cap \Gamma_{i}} vF_i$, for every
$n>i\geq 0$.
\end{itemize}

We denote by $X$ the closure of the orbit of $x_0$. Thus $(X,\sigma|_X,G)$ is a Toeplitz subshift. 
\begin{definition}
We say that a finite clopen partition $\P$ of $X$ is a regular Kakutani-Rokhlin partition (r-K-R partition),
if there exists a finite index subgroup  $\Gamma$ of $G$  with a fundamental domain $F$ containing $e$ and a clopen $C_k$, such that
$$
\P=\{\sigma^{u^{-1}}(C_{k}): u\in F, 1\leq k\leq N\}$$
and 
$$
   \sigma^{\gamma}(\bigcup_{k=1}^NC_k)=\bigcup_{k=1}^NC_k \mbox{ for every } \gamma\in\Gamma.
$$
\end{definition}

To construct a regular Kakutani-Rokhlin partition of $X$, we need the following technical lemma.

\begin{lemma}\label{refine}
Let $\P'=\{\sigma^{u^{-1}}(D_{k}): u\in F, 1\leq k\leq N\}$ be a r-K-R partition of $X$ and $\Q$ any other finite clopen partition of $X$. Then there exists a r-K-R partition $\P=\{\sigma^{u^{-1}}(C_{k}): u\in F, 1\leq k\leq M\}$ of $X$ such that 
\begin{enumerate}
\item $\P$ is finer than $\P'$ and $\Q$,

\item $\bigcup_{k=1}^MC_k=\bigcup_{k=1}^ND_k$.
\end{enumerate}
\end{lemma}
\begin{proof}
Let $F=\{u_0, u_1,\cdots, u_{|F|-1}\},$ with $u_0=e$.

We refine every set $D_k$ with respect to the partition $\Q$. Thus we get a collection of disjoint sets
$$
D_{1,1},\cdots, D_{1,l_1};\cdots ; D_{N,1},\cdots , D_{N,l_N},
$$ 
such that each of these sets is in an atom of $\Q$ and $D_k=\bigcup_{j=1}^{l_k}D_{k,j}$ for every $1\leq k\leq N$.
Thus $\P_0=\{\sigma^{u^{-1}}(D_{k,j}): u\in F, 1\leq j\leq l_k,  1\leq k\leq N\}$ is a r-K-R partition of $X$.  For simplicity we write
$$
\P_0=\{\sigma^{u^{-1}}(D_k^{(0)}): u\in F, 1\leq k\leq N_0\}.
$$
We have that $\P_0$ verifies (2) and every $D_k^{(0)}$ is contained in  atoms of $\P'$ and 
$\Q$.

Let $  0\leq n <|F|-1$. Suppose that we have defined a r-K-R partition of $X$ 
$$
\P_n=\{\sigma^{u^{-1}}(D_k^{(n)}): u\in F, 1\leq k\leq N_n\},
$$
such that $\P_n$  verifies (2) and such that for every $0\leq j\leq n$ and $1\leq k\leq N_n$ there exist $A\in\P'$ and $B\in \Q$ such that
$$
\sigma^{u_j^{-1}}(D_k^{(n)})\subseteq A, B.
$$
Now we refine every set $\sigma^{u_{n+1}^{-1}}(D_k^{(n)})$ with respect to $\Q$. Thus we get a collection of disjoint sets 
$$
D_{1,1},\cdots, D_{1,s_1};\cdots ; D_{N_n,1},\cdots , D_{N_n,s_{N_n}}
$$
such that each of these sets is in an atom of $\Q$ and $\sigma^{u_{n+1}^{-1}}(D_k^{(n)})=\bigcup_{j=1}^{s_k}D_{k,j}$, for every $1\leq k\leq N_n$. 

For every $1\leq k\leq N_n$ and $1\leq j\leq s_k$, let $C_{k,j}=\sigma^{u_{n+1}}(D_{k,j})\subseteq D_k^{(n)}$. We have that
$$
\P_{n+1}=\{\sigma^{u^{-1}}(C_{k,j}): u\in F, 1\leq j\leq s_k, 1\leq k\leq N_n\}
$$
is a r-K-R partition of $X$  verifying (2) and such that for every $0\leq i\leq n+1$, $1\leq j\leq s_k$  and $1\leq k\leq N_n$ there exist $A\in\P'$ and $B\in \Q$ such that
$$
\sigma^{u_j^{-1}}(C_{k,j})\subseteq A, B.
$$
At the step $n=|F|-1$ we get   $\P=\P_{|F|-1}$  verifying (1) and (2). 
 \end{proof}

\begin{proposition}\label{general-KR}
There exists a sequence $(\P_n=\{\sigma^{u^{-1}}(C_{n,k}): u\in F_n, 1\leq k\leq k_n\})_{n\geq 0}$ of r-K-R partitions of $X$ 
such that for every $n\geq 0$,
\begin{enumerate}
\item $\P_{n+1}$ is finer than $\P_n$,


\item $C_{n+1}\subseteq C_{n}=\bigcup_{k=1}^{k_n}C_{n,k}$,

\item $\bigcap_{n\geq 1}C_n=\{x_0\}$,


\item The  sequence $(\P_n)_{n\geq 0}$  spans the topology of $X$.
\end{enumerate}
\end{proposition}

\begin{proof}
For every $n\geq 0$, let define
$$
C_n=\{x\in X: Per(x,\Gamma_n,a)=Per(x_0,\Gamma_n,a) \forall a\in\Sigma\}.
$$
From \cite[Proposition 6]{CP} we get
$$
C_n=\overline{\{\sigma^{\gamma}(x_0):\gamma\in\Gamma_n\}},
$$
and that $\P_n'=\{\sigma^{u^{-1}}(C_n): u\in F_n\}$ is a clopen partition of $X$ such that $\sigma^{\gamma}(C_n)=C_n$ for every $\gamma\in\Gamma_n$.  Thus $\P_n'$ is a r-K-R partition of $X$.  Furthermore, the sequence $(\P_n')_{n\geq 0}$ verifies (1), (2) and (3).

For every $n\geq 0$, let $\Q_n=\{[B]\cap X: B\in\Sigma^{F_n}, [B]\cap X\neq \emptyset\}$. This is a finite clopen partition of $X$ and $(\Q_n)_{n\geq 0}$ spans the topology of $X$. 

We define $\P_{0} =\{\sigma^{u^{-1}}(C_{0,k}): u\in F_0, 1\leq k\leq k_0\})$ the r-K-R partition finer than $\P_0'$ and $\Q_0$ given by Lemma \ref{refine}.  Now we take $\P_n''$ the r-K-R partition finer that $\P_{n-1}$ and $\Q_n$ given by Lemma \ref{refine}, and we define
$$
\P_n=\{\sigma^{u^{-1}}(C_{n,k}): u\in F_n, 1\leq k\leq k_n\},
$$
the r-K-R partition finer than $\P'=\P'_n$ and $\Q=\P_n''$ given by Lemma \ref{refine}. Thus $\P_n$ is finer than $\P_{n-1}$ and $\Q_n$. This implies that the sequence $(\P_n)_{n\geq 0}$ verifies (1) and (4).  Since $\bigcup_{k=1}^{k_n}C_{n,k}=C_n$, we deduce that $(\P_n)_{n\geq 0}$ verifies (2) and (3).
\end{proof}

\begin{remark}
{\rm  The sequence of partitions of Proposition \ref{general-KR} is a generalization to  Toeplitz $G$-subshifts    of  the sequences of  Kakutani-Rokhlin partitions for   Toeplitz $\ZZ$-subshifts  introduced in \cite{GJ}.  See \cite{HPS} for more details about Kakutani-Rokhlin partitions for minimal $\ZZ$-actions on the Cantor set}
\end{remark}

\begin{definition} We say that a sequence $(\P_n)_{n\geq 0}$ of r-K-R partitions as in Proposition \ref{general-KR} is a nested sequence of  r-K-R partitions of $X$.
\end{definition}
Let  $(\P_n=\{\sigma^{u^{-1}}(C_{n,k}): u\in F_n, 1\leq k\leq k_n\})_{n\geq 0}$ be a sequence of nested r-K-R partitions of $X$.

For every $n\geq 0$ we define the matrix  $M_n\in
\M_{k_n\times k_{n+1}}(\ZZ^+)$ as
$$
M_n(i,k)=|\{\gamma\in F_{n+1}\cap\Gamma_n:  \sigma^{\gamma^{-1}}(C_{n+1,k})\subseteq C_{n,i}\}|,
$$
We call $M_n$ the {\it incidence matrix} of the partitions $\P_{n+1}$ and $\P_n$.

Let $p$ be a positive integer. For every $n\geq 1$ we denote by
$\triangle(n,p)$ the closed convex hull generated by the vectors
$\frac{1}{p}e_1^{(n)},\cdots,\frac{1}{p}e_n^{(n)}$, where
$e_1^{(n)},\cdots,e_n^{(n)}$ is the canonical base in  $\RR^n$.
Thus $\triangle(n,1)$ is the unitary simplex in $\RR^n$.

Observe that for every $n\geq 0$ and $1\leq k\leq k_{n+1}$,
$$\sum_{i=1}^{k_n}M_n(i,k)=\frac{|F_{n+1}|}{|F_n|}.$$
This implies that $M_n(\triangle(k_{n+1},|F_{n+1}|))\subseteq \triangle(k_n, |F_n|)$.

The next result characterizes the  maximal equicontinuous factor, the space of invariant probability measures  and the associated ordered group of $(X,\sigma|_X,G)$ in terms of the sequence of incidence matrices of a nested sequence of  r-K-R partitions.

\begin{proposition} \label{measures and odometer}
 Let  $(\P_n=\{\sigma^{u^{-1}}(C_{n,k}): u\in F_n, 1\leq k\leq k_n\})_{n\geq 0}$ be a nested sequence of r-K-R partitions of $X$ with an associated sequence of incidence matrices $(M_n)_{n\geq 0}$. Then 
\begin{enumerate}
\item   $(X,\sigma|_X,G)$ is an almost 1-1
extension of the odometer $O=\varprojlim_n(G/\Gamma_n,\pi_n)$,

\item there is an affine homeomorphism between the set of
invariant probability measures of $(X,\sigma|_{X},G)$ and the
inverse limit  $\varprojlim_n(\triangle(k_n, |F_n|),M_n)$,

\item the ordered group $\G(X,\sigma|_X,G)$ is isomorphic to
$(H/inf(H),(H/inf(H))^+,u + inf(H))$, where $(H,H^+)$ is given by
$$
\xymatrix{{\ZZ^{}}\ar@{->}[r]^{M^T} &{\ZZ^{k_0}}\ar@{->}[r]^{M^T_0} & {\ZZ^{k_1}}\ar@{->}[r]^{M_1^T} &
{\ZZ^{k_2}}\ar@{->}[r]^{M^T_2}& {\cdots}},
$$
where $M=|F_0|(1,\cdots,1)$ and $u=[M^T,0]$.
\end{enumerate}
\end{proposition}

\begin{proof}
{\bf 1.}  For every $x\in X$ and $n\geq 0$, let $v_n(x)\in F_n$ be such that $x\in \sigma^{v_n(x)^{-1}}(C_n)$.

The map $\pi: X\to O$ given by
$\pi(x)=(v_n(x)^{-1}\Gamma_n)_{n\geq 1}$ is well defined, is a factor map
and verifies $\pi^{-1}(\pi(x_0))=\{x_0\}$. This shows that
$(X,\sigma|_X,G)$ is an almost 1-1 extension of $O$.

\medskip

{\bf 2.} It is clear that for any invariant probability measure $\mu$ of
$(X,\sigma|_{X},G)$, the sequence $(\mu_n)_{n\geq 0}$, with
$\mu_n=(\mu(C_{n,k}): 1\leq k\leq k_n)$, is an element of the
inverse limit $\varprojlim_n(\triangle(k_n, |F_n|),M_n)$. Conversely, any element
$(\mu_{n,k}: 1\leq k\leq k_n)_{m\geq 0}$ of such inverse limit,
defines a probability measure $\mu$ on the $\sigma$-algebra
generated by $(\P_n)_{n\geq 0}$, which is equal to the Borel $\sigma$-algebra of $X$ because $(\P_n)_{n\geq 0}$ spans the topology
of $X$ and is countable. Since the sequence $(F_n)$ is F\o lner,  it is standard to check that the measure $\mu$ is invariant
by  the $G$-action. 

The function $\mu\mapsto (\mu_n)_{n\geq 0}$  is thus
an affine bijection between $\M(X,\sigma|_X,G)$ and the inverse
limit $\varprojlim_n(\triangle(k_n, |F_n|),M_n)$. Observe that this function  is a homeomorphism  with respect to the weak topology in
$\M(X,\sigma|_X,G)$ and the product topology in the inverse limit.

\medskip

{\bf 3.}  We denote by $[k,-1]$ the class of the element $(k,-1)\in \ZZ\times\{-1\}$ in $H$.

Let $\phi:H \to D_m(X,\sigma|_{X},G)$ be the function
given by $\phi([v,n])=\sum_{k=1}^{k_n}v_i[1_{C_{n,k}}]$, for every
$v=(v_1,\cdots, v_{k_n})\in \ZZ^{k_n}$ and $n\geq 0$, and $\phi([k,-1])=k1_X$ for every $k\in\ZZ$. It is easy
to check that $\phi$ is a well defined homomorphism of groups that
verifies  $\phi(H^+)\subseteq D_m(X,\sigma|_X,G)^+$. Since
$(\P_n)_{n\geq 0}$ spans the topology of $X$, every function $f\in
C(X,\ZZ)$ is constant on every atom of $\P_n$, for some $n\geq 0$.
This implies that $\phi$ is surjective. Lemma \ref{nexo} and (2) of Proposition \ref{measures and odometer},
imply that $Ker(\phi)=inf(H)$.  Finally,  $\phi$ induces a
isomorphism $\widehat{\phi}:H/inf(H)\to D_m(X,\sigma|_X,G)$  such
that $\widehat{\phi}((H/inf(H))^+)=D_m(X,\sigma|_X,G)^+$.  Since $[1,-1]=[M^T,0]$, we get $\phi([M^T,0])=[1_X]$.
\end{proof}

\section{Kakutani-Rokhlin partitions with  prescribed incidence matrices.}\label{realization-ordered-groups}

We say that a sequence of positive integer matrices  $(M_n)_{n\geq 0}$ is {\it managed} by the
increasing sequence of positive integers $(p_n)_{n\geq 0}$, if for
every $n\geq 0$ the integer $p_n$ divides $p_{n+1}$, and if the
matrix $M_n$ verifies the following properties: 
\begin{enumerate}
\item $M_n$ has $k_n\geq 2$ rows and $k_{n+1}\geq 2$ columns;

\item $\sum_{i=1}^{k_n}M_n(i,k)=\frac{p_{n+1}}{p_n}$, for every
$1\leq k\leq k_{n+1}$.

\end{enumerate}



If $(M_{n})_{n\geq 0}$ is a   sequence of
matrices  managed by $(p_n)_{n\geq 0}$, then for each $n\geq 0$,\\ 
$M_n(\triangle(k_{n+1},p_{n+1}))\subseteq \triangle(k_n,p_n)$.

Observe that the sequences of incidence matrices  associated to the nested sequences of  r-K-R partitions defined in Section \ref{K-R partitions}  are   managed by $(|F_n|)_{n\geq 0}$.

In this Section we construct Toeplitz subshifts with nested sequences of r-K-R partitions whose sequences of incidence matrices are managed.

\subsection{Construction of the partitions}
In the rest of this Section  $G$ is an amenable and residually finite group. 
Let $(\Gamma_n)_{n\geq 0}$  be a decreasing sequence of finite index normal subgroup of $G$ such that $\bigcap_{n\geq 0}\Gamma_n=\{e\}$, and let $(F_n)_{n\geq 0}$ be a F\o lner sequence of $G$ such that
\begin{itemize}
\item[(F1)] $\{e\}\subseteq F_n\subseteq F_{n+1}$ and  $F_{n}$ is a
fundamental domain of $G/\Gamma_{n}$, for every $n\geq 0$.

\item[(F2)] $G=\bigcup_{n\geq 0}F_{n}.$

\item[(F3)] $F_{n}=\bigcup_{v\in F_{n}\cap \Gamma_{i}} vF_i$, for every
$n>i\geq 0$.
\end{itemize}
Lemma \ref{folner2} ensures the existence of a F\o lner sequence verifying conditions (F1), (F2) and (F3).




For every $n\geq 0$, we call $R_n$ the set $F_n\cdot F_n^{-1}\cup F_n^{-1}\cdot F_n$. This will enable us to define a ``border'' of each domain $F_{n+1}$.  

Let $\Sigma$ be a finite alphabet.  For every $n\geq 0$, let
$k_n\geq 3$ be an integer. We say that the sequence of sets
$(\{B_{n,1},\cdots,B_{n,k_n}\})_{n\geq 0}$ where for any $n \ge 0$, 
 $\{B_{n,1},\cdots,B_{n,k_n}\}\subseteq \Sigma^{F_n}$ is a collection of different functions, 
{\em verifies conditions} (C1)-(C4) if it verifies the following four conditions for any $n \ge 0$:
\begin{itemize}
\item[(C1)] $\sigma^{\gamma^{-1}}(B_{n+1,k})|_{F_{n}}\in \{B_{n,i}:1\leq i\leq
k_{n}\}$, for every $\gamma\in F_{n+1}\cap \Gamma_{n}$, $1\leq k\leq
k_{n+1}$.

\item[(C2)] $B_{n+1,k}|_{F_{n}}=B_{n,1}$, for every $1\leq k\leq
k_{n+1}$.

\item[(C3)]  For  any $g\in F_n$ such that for
some $1\leq k, k'\leq k_n$, $B_{n,k}(gv)=B_{n,k'}(v)$ for all
$v\in F_n \cap g^{-1}F_n $, then $g=e$.

\item[(C4)] $\sigma^{\gamma^{-1}}(B_{n+1,k})|_{F_{n}}=B_{n,k_{n}}$
for every $\gamma\in (F_{n+1}\cap \Gamma_{n}) \cap \left[F_{n+1}\setminus
F_{n+1}g^{-1} \right]$, for some $g\in R_{n}$.
\end{itemize}

\begin{example}{\rm 
To illustrate these conditions, let us consider the case $G =\ZZ$, $\Sigma= \{ 1,2,3,4\}$ and $\Gamma_n=3^{2(n+1)}\ZZ$ 
for every $n\geq 0$.  The set
$$
 F_n=\left\{ -\left(\frac{3^{2(n+1)}-1}{2}  \right ), -\left(\frac{3^{2(n+1)}-1}{2}  \right ) +1 , \cdots, \left(\frac{3^{2(n+1)}-1}{2} \right )  \right\}
$$
is a  fundamental domain of $\ZZ/\Gamma_n$.  Furthermore we have
 $$
 F_n=\bigcup_{v\in\{k3^{2n}: -4\leq k\leq 4\}}(F_{n-1}+v),
 $$
for every $n\geq 1$. This shows that sequence $(F_n)_{n\geq 0}$ satisfies (F1), (F2) and (F3).
 
Now let us consider the case where $k_n=4$ for every $n\geq 0$. We define $B_{0,k}(j)=k$ for every $j\in F_0$ and $1\leq k\leq 4$, and for $n\geq 1$,
 $$
 B_{n,k}|_{F_{n-1}}=B_{n-1,1}, \hspace{2mm}  B_{n,k}|_{F_{n-1}+v}=B_{n-1,4} \mbox{ for } v\in \{-l\cdot 3^{2n}, l\cdot 3^{2n}: l=3,4\}.
 $$ 
Thus they verify the  conditions (C1) and (C4). We fill the rest of the $B_{n,k}|_{F_{n-1}+v}$ with $B_{n-1,3}$ and $B_{n-1,2}$ in order that 
 $B_{n,1},\cdots, B_{n,4}$ are different.   
They satisfy conditions (C2) and (C4). The limit in $\Sigma^\ZZ$ of the functions $B_{n,1}$ is a $\ZZ$-Toeplitz sequence $x$. 
If $X$ denotes the closure of the orbit of $x$, then we prove in the next lemma (in a more general setting) that  
$$
(\P_n=\{\sigma^{j}([B_{n,k}]\cap X): j\in F_n, 1\leq k\leq 4\})_{n\geq 0}
$$
is a sequence of nested Kakutani-Rokhlin partitions of the subshift $X$.}
\end{example}

In the next lemma, we show that conditions (C1) and (C2) are
sufficient to construct a Toeplitz sequence. The technical
conditions (C3) (aperiodicity) and (C4) (also known as ``forcing
the border'') will allow to construct a  nested sequence of    r-K-R
partitions of $X$. 

\begin{lemma}\label{sufficient-condition}
 Let $(\{B_{n,1},\cdots,B_{n,k_n}\})_{n\geq 0}$ be a sequence that
 verifies conditions (C1)-(C4). Then:

\begin{itemize}
\item[(1)] The set $\bigcap_{n\geq 0}[B_{n,1}]$ contains only one
element $x_0$  which  is a Toeplitz sequence.

\item[(2)] Let $X$ be the  orbit closure of $x_0$ with respect to the
shift action. For every $n\geq 0$, let
$$
\P_n=\{\sigma^{u^{-1}}([B_{n,k}]\cap X): 1\leq k\leq k_n, u\in
F_n\}.
$$
Then $(\P_n)_{n\geq 0}$ is a sequence of nested r-K-R partitions of $X$. 
\end{itemize}

Let $(M_n)_{n\geq 0}$ be the sequence of incidence matrices of $(\P_n)_{n\geq 0}$.
Thus we have 

\begin{itemize}

\item[(3)] The Toeplitz subshift $(X,\sigma|_X,G)$ is an almost 1-1
extension of the odometer $O=\varprojlim_n(G/\Gamma_n,\pi_n)$.

\item[(4)] There is an affine homeomorphism between the set of
invariant probability measures of $(X,\sigma|_{X},G)$ and the
inverse limit  $\varprojlim_n(\triangle(k_n, |F_n|),M_n)$.
\item[(5)] The ordered group $\G(X,\sigma|_X,G)$ is isomorphic to
$(H/inf(H),(H/inf(H))^+,u + inf(H))$, where $(H,H^+)$ is given by
$$
\xymatrix{{\ZZ^{}}\ar@{->}[r]^{M^T} &{\ZZ^{k_0}}\ar@{->}[r]^{M^T_0} & {\ZZ^{k_1}}\ar@{->}[r]^{M_1^T} &
{\ZZ^{k_2}}\ar@{->}[r]^{M^T_2}& {\cdots}},
$$
with $M=|F_0|(1,\cdots,1)$ and $u=[M^T,0]$.
\end{itemize}
\end{lemma}

\begin{proof}
 Condition (C2) implies that $\bigcap_{n\geq 0}[B_{n,1}]$
is non empty, and since $G=\bigcup_{n\geq 0} F_n$, there is only
one element $x_0$ in  this intersection.  Let $X$ be the orbit
closure  of $x_0$. For every $n\geq 0$ and $1\leq k\leq k_n$, we
denote $C_{n,k}=[B_{n,k}]\cap X$.

\medskip

{\it Claim:} For every $m>n\geq 0$, $1\leq k\leq k_m$ and $\gamma\in F_m\cap\Gamma_n$,  
\begin{equation}\label{inclusion}
\sigma^{\gamma^{-1}}(B_{m,k})|_{F_{n}}\in\{ B_{n,i}: 1\leq i\leq k_{n}\}.
\end{equation}

\medskip

Condition (C1) implies that (\ref{inclusion}) holds when $n=m-1$. We will show the claim by induction on $n$.

Suppose that for every   $1\leq k\leq k_m$ and $\gamma\in F_m\cap\Gamma_{n+1}$,  
$$
\sigma^{\gamma^{-1}}(B_{m,k})|_{F_{n+1}}\in\{ B_{n+1,i}: 1\leq i\leq k_{n+1}\}.
$$
Let $g\in \Gamma_n\cap F_m$. Condition (F3) implies there exist $v\in\Gamma_{n+1}\cap F_m$ and $u\in F_{n+1}$ such that $g=vu$. Thus we get
$$
\sigma^{g^{-1}}(B_{m,k})|_{F_n}=\sigma^{u^{-1}v^{-1}}(B_{m,k})=\sigma^{v^{-1}}(B_{m,k})|_{uF_n}.
$$
Since $u\in \Gamma_n\cap F_{n+1}$, condition (F3) implies that $uF_n\subseteq F_{n+1}$. Then by hypothesis, there exists $1\leq l\leq k_{n+1}$ such that
$$
\sigma^{v^{-1}}(B_{m,k})|_{uF_n}=B_{n+1,l}|_{uF_n},
$$ 
which is equal to some $B_{n,s}$, by (C1). This shows the claim.

\medskip

From (\ref{inclusion})  we
deduce that $\sigma^{\gamma^{-1}}(x_0)|_{F_n}\in\{ B_{n,i}: 1\leq i\leq k_n\}$, for
every $\gamma\in \Gamma_n$. Thus if $g$ is any element in $G$,
and $u\in F_n$ and $\gamma\in \Gamma_n$ are such that $g=\gamma u$, then

$\sigma^{g^{-1}}(x_0)=\sigma^{u^{-1}}(\sigma^{\gamma^{-1}}(x_0))\in
\sigma^{u^{-1}}(C_{n,k})$, for some $1\leq k\leq k_n$. It follows that
$$
\P_n=\{\sigma^{u^{-1}}(C_{n,k}): 1\leq k\leq k_n, u\in
F_n\}
$$
is a clopen covering of $X$.

From  condition (C2) and (\ref{inclusion}) we get that  $\sigma^{\gamma^{-1}}(x_0)|_{F_{n-1}}  = B_{n-1,1}$ for any $\gamma \in \Gamma_n$, which implies that $F_{n-1}\subseteq Per(x_0,\Gamma_n)$.
This shows that $x_0$ is Toeplitz.

Now we will show that $\P_n$ is a partition.  Suppose that $1\leq k,l\leq k_n$ and $u\in F_n$ are such that $\sigma^{u^{-1}}(C_{n,k})\cap C_{n,l}\neq \emptyset$.  Then there exist $x\in C_{n,k}$ and $y\in C_{n,l}$ such that $\sigma^{u^{-1}}(x)=y$.  From this we have $x(uv)=y(v)$ for every $v\in G$. In particular, $x(uv)=y(v)$ for every $v\in F_n\cap u^{-1}F_n$, which implies $B_{n,k}(uv)=B_{n,l}(v)$  for every $v\in F_n\cap u^{-1}F_n$. From condition (C3)  we get $u=e$ and $k=l$. This  
 ensures that the {\em set of return times} of $x_0$ to
$\bigcup_{k=1}^{k_n}C_{n,k}$, {\em i.e.} the set $\{g \in G: \sigma^{g^{-1}}(x_0) \in \bigcup_{k=1}^{k_n}C_{n,k}  \}$, is $\Gamma_n$. From this it follows that $\P_n$ is a r-K-R  partition. From (C1) we have that $\P_{n+1}$ is finer than $\P_n$ and that $C_{n+1}\subseteq \bigcup_{k=1}^{k_n}C_{n,k}=C_n$. By the definition of $x_0$ we have that $\{x_0\}=\bigcap_{n\geq 0} C_n$.

Now we will show that $(\P_n)_{n\geq 0}$ spans the topology of $X$.
Since every $\P_n$ is a partition, for every $n\geq 0$
and every $x\in X$ there are unique $v_n(x)\in F_n$ and $1\leq
k_n(x) \leq k_n$ such that
$$
x\in \sigma^{v_n(x)^{-1}}(C_{n,k_n(x)}).
$$
The collection $(\P_n)_{n\geq 0}$ spans the topology of $X$ if and
only if $(v_n(x))_{n\geq 0}=(v_n(y))_{n\geq 0}$ and
$(k_n(x))_{n\geq 0}=(k_n(y))_{n\geq 0}$  imply $x=y$.

Let $x, y\in X$ be two sequences such that  $v_n(x)=v_n(y)=v_n$
and $k_n(x)=k_n(y)$ for every $n\geq 0$. Let $g\in G $ be such that
$x(g) \neq y(g)$.

We have then for any  $n\ge 0$
$$\sigma^{v_n}(x)|_{F_n} = \sigma^{v_n}(y)|_{F_n} \in\{B_{n,i}:1\leq i\leq
k_{n}\}, $$
and then
$$x|_{v_n^{-1}F_n} =y|_{v_n^{-1}F_n}.$$
Thus by definition, we get $g \not\in v_n^{-1}F_n$ for any $n$. We
can take $n$ sufficiently large in order that $g\in F_{n-1}$.

Let $\gamma\in \Gamma_n$ and $u\in F_n$ such that $v_n(x)g=\gamma
u$.  Observe that $ug^{-1}\notin F_n$. Indeed, if $ug^{-1}\in
F_n$, then the relation $v_n(x)=\gamma ug^{-1}$ implies
$\gamma=e$, but in that case we get $v_n(x)g=u\in F_n$ which is
not possible by hypothesis.  By the condition (C1), there exists an index $1 \le i\le k_n$ such that
$ \sigma^{\gamma^{-1}}(\sigma^{v_n}(x))|_{F_n} = B_{n,i}$ and then
$$x(g) = \sigma^{\gamma^{-1}}\sigma^{v_n}(x)(\gamma^{-1} v_n g) = B_{n,i} (u).$$

Let $\gamma'\in \Gamma_{n-1}\cap F_n$
and $u'\in F_{n-1}$ such that $u=\gamma'u'$. Since
$\gamma'u'g^{-1}=ug^{-1}\notin F_n$, we get $\gamma'\in
F_n\setminus F_ngu'^{-1}$. This implies that $\gamma'\in
F_n\setminus F_nw$, for $w=gu'^{-1}\in R_{n-1}$ and $B_{n,i} (u)= B_{n-1,k_{n-1}}(u')$ by the condition (C4).
Thus $x(g)=B_{n-1,k_{n-1}}(u')$. The same argument implies that $y(g)=B_{n-1,k_{n-1}}(u')=x(g)$ and we obtain a contradiction.

This shows that $(\P_n)_{n\geq 0}$ is a sequence of nested r-K-R partitions of $X$. 

The point (3), (4) and (5) follows from Propositions \ref{measures and odometer}.
\end{proof}

The next result shows that, up to telescope a managed sequence
of matrices, it is possible to obtain a managed sequence of matrices with
sufficiently large coefficient to satisfy the conditions of Lemma
\ref{sufficient-condition}.

\begin{lemma}\label{subsequence}  
Let $(M_n)_{n\geq 0}$ be a sequence of matrices managed by $(|F_n|)_{n\geq 0}$. Let $k_n$ be the number of rows of $M_n$, for every $n\geq 0$.


Then there exists an increasing sequence $(n_i)_{i\geq 0}\subseteq
\NN$ such that for every $i\geq 0$ and every $1\leq k \leq
k_{n_{i+1}}$,
\begin{itemize}
\item[(i)] $R_{n_i}\subseteq F_{n_{i+1}}$,

\item[(ii)] For every $1\leq l\leq k_{n_i}$,
$$
M_{n_i}M_{n_i+1}\cdots M_{n_{i+1}-1}(l,k)>
 1+|\bigcup_{g\in R_{n_i}}F_{n_{i+1}}\setminus
F_{n_{i+1}}g^{-1}|
$$
\end{itemize}
If in addition there exists a constant $K>0$ such that  $k_{n+1}\leq K\frac{|F_{n+1}|}{|F_n|}$ for every $n\geq 0$, then the sequence $(n_i)_{i\geq 0}$ can be chosen in order that
\begin{itemize}
\item[(iii)] $k_{n_{i+1}}< M_{n_i}\cdots M_{n_{i+1}-1}(i,k),$ for
every $1\leq i\leq k_{n_i}$.
\end{itemize}
\end{lemma}

\begin{proof}
We define $n_0=0$. Let $i\geq 0$ and suppose that we have defined
$n_j$ for every $0\leq j\leq i$. Let $m_0>n_i$ be such that for
every $m\geq m_0$,
$$
R_{n_i}\subseteq F_m.
$$
Let $0<\varepsilon<1$ be such that $\varepsilon|R_{n_i}|<1$. Since
$(F_n)_{n\geq 0}$ is a F\o lner sequence, there exists $m_1>m_0$
such that for every $m\geq m_1$,
\begin{equation}\label{eq2-subsequence}
\frac{|F_m\setminus F_m
g^{-1}|}{|F_m|}<\frac{\varepsilon}{|F_{n_i+1}|}, \mbox{ for every
} g\in R_{n_i}.
\end{equation}
Since $\varepsilon|R_{n_i}|<1$, there exists $m_2>m_1$ such that
for every $m\geq m_2$,
$$
1- \frac{|F_{n_i+1}|}{|F_m|}>\varepsilon|R_{n_i}|.
$$
Then
$$
\frac{|F_m|}{|F_{n_i+1}|}-1>\varepsilon|R_{n_i}|\frac{|F_m|}{|F_{n_i+1}|}.
$$
Since the matrices $M_n$ are positive, using induction on $m$ and condition (2)  for managed sequences, we get
$$
M_{n_i}\cdots M_{m-1}(l,j)\geq \frac{|F_m|}{|F_{n_i+1}|}, \mbox{
for every } 1\leq l\leq k_{n_i}, 1\leq j\leq k_{m}.
$$
Combining the last two equations we get
$$
M_{n_i}\cdots
M_{m-1}(l,j)-1>\varepsilon|R_{n_i}|\frac{|F_m|}{|F_{n_i+1}|},
$$
and from equation (\ref{eq2-subsequence}), we obtain
$$
M_{n_i}\cdots M_{m-1}(l,j)-1>|F_m\setminus F_m
g^{-1}||R_{n_i}|, \mbox{ for every } g\in R_{n_i},
$$
which finally implies that
$$
M_{n_i}\cdots M_{m-1}(l,j)>|\bigcup_{g\in R_{n_i}} F_m\setminus
F_mg^{-1}|+ 1, \mbox{ for every } 1\leq l\leq k_{n_i},
1\leq j\leq k_{m}.
$$
Now, suppose there exists $K>0$ such that  $k_{m+1}\leq K\frac{|F_{m+1}|}{|F_m|}$ for every $m\geq 0$.  The property (2) for managed sequences of matrices implies  
$$
M_{n_i}\cdots M_m(l,j)\geq \frac{|F_{m+1}|}{|F_{n_i+1}|} \mbox{ for every } m> n_{i}.
$$
Let $m_3>m_2$ be such that  $K< \frac{|F_{m}|}{|F_{n_i+1}|}$ for every $m\geq m_3$. Then for every $m\geq m_3$ we have
$$
k_{m+1}\leq K\frac{|F_{m+1}|}{|F_{n_i}|}\leq  M_{n_i}\cdots M_{m}(l,j) \mbox{ for every } 1\leq
l\leq k_{n_i} \mbox{ and } 1\leq j\leq k_{m+1}.
$$
By taking $n_{i+1}\geq m_3$ we get the desired subsequence
$(n_i)_{i\geq 0}\subseteq \NN$.
\end{proof}

The following proposition shows that given a managed sequence,
there exists a sequence of decorations verifying conditions
(C1)-(C4). The aperiodicity condition (C3) is obtained by
decorating the center of $F_n$ in a unique way with respect to
other places in $F_n$. A restriction on the number of columns of
the matrices  gives enough choices of
coloring to ensure conditions (C3) and (C4).

\begin{proposition}\label{construction}
Let  $(M_n)_{n\geq 0}$ be a sequence of matrices which is
managed by $(|F_n|)_{n\geq 0}$. For every $n\geq
0$, we denote by $k_n$  the number of rows of $M_n$.  Suppose in addition there exists $K>0$ such that $k_{n+1}\leq K\frac{|F_{n+1}|}{|F_n|}$, for every $n\geq 0$.  Then there
exists a Toeplitz subshift $(X,\sigma|_X, G)$ verifying the
following three conditions:
\begin{enumerate}
\item The set of invariant probability measures of $(X,\sigma|_X,
G)$ is affine homeomorphic to
$\varprojlim_n(\triangle(k_n,|F_n|),M_n)$.
\item The ordered group $\G(X,\sigma|_X,G)$ is isomorphic to
$(H/inf(H), (H/inf(H))^+,u+inf(H))$, where $(H,H^+)$ is given by 
$$
\xymatrix{{\ZZ^{}}\ar@{->}[r]^{M^T} &{\ZZ^{k_0}}\ar@{->}[r]^{M^T_0} & {\ZZ^{k_1}}\ar@{->}[r]^{M_1^T} &
{\ZZ^{k_2}}\ar@{->}[r]^{M^T_2}& {\cdots}},
$$
with $M=|F_0|(1,\cdots,1)$ and $u=[M^T,0]$.
 
 \item $(X,\sigma|_X, G)$ is an almost 1-1 extension of the
odometer $O=\varprojlim_n(G/\Gamma_n,\pi_n)$.
\end{enumerate}
\end{proposition}

\begin{proof} 
Let $(n_i)_{i\geq 0}\subseteq \NN$ be a sequence as in Lemma
\ref{subsequence}. Since $(M_n)_{n\geq 0}$ and the sequence
$(M_{n_i}\cdots M_{n_{i+1}-1})_{i\geq 0}$ define the same inverse
and direct  limits, without loss of generality we can assume that
for every $n\geq 0$ we have:  
$$
R_n\subseteq F_{n+1}, 
$$
$$
M_n(i,k)> 1+|\bigcup_{g\in R_{n}}F_{n+1}\setminus
F_{n+1}g^{-1}| \mbox{ for every } 1\leq i\leq k_n,  1\leq k\leq
k_{n+1},
$$
and
$$
k_{n+1}< \min\{M_n(i,j): 1\leq i\leq k_n, 1\leq j\leq k_{n+1}\}.
$$
Let $\tilde{M}$ be the $1\times (k_{0}+1)$-dimensional matrix given by 
$$
\tilde{M}(\cdot,1)=\tilde{M}(\cdot,2)=M(\cdot,1),
$$
and $\tilde{M}(\cdot,k+1)=M(\cdot,k)$ for every $2\leq k\leq k_0$.
For every $n\geq 0$,  consider the $(k_n+1)\times(k_{n+1}+1)$-dimensional matrix given by
$$
\tilde{M}_n(\cdot,1)=\tilde{M}_n(\cdot, 2)=\left( \begin{array}{c}
                                                                                1\\
                         M_n(1,1)-1\\
M_n(2,1)\\
\vdots\\
M_n(k_n,1)
                                                                                 \end{array}
\right)
$$
and
$$
\tilde{M}_n(\cdot,k+1)=\left( \begin{array}{c}
1\\
M_n(1,k)-1\\
M_n(2,k)\\
\vdots\\
M_n(k_n,k)\\                                            
\end{array}
    \right) \mbox{ for every } 2\leq k\leq k_{n+1}.
$$
Lemma \ref{nexo2} implies that the dimension groups with unit given by
$$
\xymatrix{{\ZZ^{}}\ar@{->}[r]^{M^T} &{\ZZ^{k_0}}\ar@{->}[r]^{M^T_0} & {\ZZ^{k_1}}\ar@{->}[r]^{M_1^T} &
{\ZZ^{k_2}}\ar@{->}[r]^{M^T_2}& {\cdots}},
$$
and
$$
\xymatrix{{\ZZ^{}}\ar@{->}[r]^{\tilde{M}^T} &{\ZZ^{k_0+1}}\ar@{->}[r]^{\tilde{M}^T_0} & {\ZZ^{k_1+1}}\ar@{->}[r]^{\tilde{M}_1^T} &
{\ZZ^{k_2+1}}\ar@{->}[r]^{\tilde{M}^T_2}& {\cdots}},
$$
are isomorphic. 

Thus from Lemma \ref{nexo} we get that $\varprojlim_n(\triangle(k_n,|F_n|),M_n)$ and $\varprojlim_n(\triangle(k_n+1,|F_n|),\tilde{M}_n)$ are affine homeomorphic.  Observe that  $(\tilde{M}_n)_{\geq 0}$ is managed by $(|F_n|)_{n\geq 0}$ and verifies
for every $n\geq 0$:
$$
\tilde{M}_n(i,k)\geq 1+|\bigcup_{g\in R_{n}}F_{n+1}\setminus
F_{n+1}g^{-1}| \mbox{ for every } 2\leq i\leq k_n+1,  1\leq k\leq
k_{n+1}+1,
$$
and
$$
3\leq k_{n+1}+1\leq \min\{M_n(i,j): 2\leq i\leq k_n+1, 1\leq j\leq k_{n+1}+1\}.
$$
Thus,  by Lemma \ref{sufficient-condition}, to prove the proposition it is enough to find a Toeplitz subshift having a sequence of r-K-R-partitions whose sequence of incidence matrices is $(\tilde{M}_n)_{n\geq 0}$.

\medskip

For every $n\geq 0$, we call $l_n$ and $l_{n+1}$ the number of rows and columns of $\tilde{M}_n$ respectively.
 
\medskip

For every $n\geq 0$, we will construct a collection of functions
$B_{n,1},\cdots, B_{n,l_n}\in \Sigma^{F_n}$ as in Lemma
\ref{sufficient-condition}, where   $\Sigma=\{1,\cdots, l_0\}$. 

 For every $1\leq k\leq l_0$ we define $B_{0,k}\in\Sigma^{F_0}$ by $B_{0,k}(g)=k$, for every $g\in F_0$. Observe that the   collection
$\{B_{0,1},\cdots, B_{0,l_0}\}$ verifies condition (C3).



\medskip

Let $n\geq 0$. Suppose that we have defined $B_{n,1},\cdots,
B_{n,l_n}\in \Sigma^{F_n}$ verifying condition (C3). For $1\leq
k\leq l_{n+1}$, we define
$$
B_{n+1,k}|_{F_n}=B_{n,1},
$$
and
$$
\sigma^{s^{-1}}(B_{n+1,k})|_{F_n}=B_{n,l_n} \mbox{ for every } s\in  
\bigcup_{g\in R_{n}}F_{n+1}\setminus F_{n+1}g^{-1}\cap \Gamma_n.
$$
We fill the rest of the coordinates $v\in F_{n+1}\cap \Gamma_n$ in
order that $\sigma^{v^{-1}}(B_{n+1,k})|_{F_n}\in \{B_{n,1},\cdots, B_{n,l_n}\}$
and such that
$$|\{v\in F_{n+1}\cap \Gamma_n: \sigma^{v^{-1}}(B_{n+1,k})|_{F_n}=B_{n,i}\}|=\tilde{M}_n(i,k),$$
 for every $2\leq i\leq
l_n$. 

Since $\tilde{M}_n(1,k)=1$,  if $\sigma^{v^{-1}}(B_{n+1,k})|_{F_n}=B_{n,1}$ then $v=e$.


Notice that the number of $B\in\Sigma^{F_{n+1}}$ that we could choose to be equal to $B_{n+1,k}$ is at least $\tilde{M}_n(2,k)+1$, because there are at least $\tilde{M}_n(2,k)+1$ free coordinates to be filled with $\tilde{M}_{n}(2,k)$ copies of
$B_{n,2}$ and  one copy of 
$B_{n,l_n}$. Since  $\tilde{M}_n(2,k)+1\geq l_{n+1}$, the number of columns of $\tilde{M}_n$ which are equal to $\tilde{M}_n(\cdot,k)$ does not exceed the  number of possible choices of functions in $\Sigma^{F_{n+1}}$ in order that 
$B_{n+1,1},\cdots, B_{n+1,l_{n+1}}$ are pairwise different.  

By construction, every function $B_{n+1,k}$ verifies (C1), (C2)
and (C4). Let us assume there are $g\in F_{n+1}$ and $1\leq
k,k'\leq k_{n+1}$ such that $B_{n+1,k}(gv)=B_{n+1,k'}(v)$ for any
$v$ where it is defined, then by the induction hypothesis, $g\in
\Gamma_n$. This implies
$\sigma^{g^{-1}}(B_{n+1,k})|_{F_n}=B_{n+1,k'}|_{F_n}=B_{n,1}$ and then $g=e$. 
This shows that the collection $B_{n+1,1},\cdots,
B_{n+1,l_{n+1}}$ verifies (C3). We conclude applying Lemma
\ref{sufficient-condition}.
\end{proof}

For positive integers $n_1,\cdots, n_k$,  we denote by $(n_1,\cdots, n_k)!$  the corresponding  multinomial coefficient. That is,
$$
(n_1,\cdots, n_k)!=\frac{(n_1+\cdots + n_k)!}{n_1!\cdots n_k!}.
$$

\begin{remark}\label{remark-bound}
{\rm  In Proposition \ref{construction}, to construct the collection of functions $(B_{n,1}\cdots, B_{n,l_n})_{n\geq 0}$ we just need that the number of columns  of $\tilde{M}_n$ which are equal to $\tilde{M}_n(\cdot,k)$ does not exceed the number of possible ways to construct different functions $B\in \Sigma^{F_n}$ verifying $B|_{F_{n-1}}=B_{n-1,1}$ and $B|_{vF_{n-1}}=B_{n-1,l_{n-1}}$ for every $v \in  
\bigcup_{g\in R_{n-1}}F_{n}\setminus F_{n}g^{-1}\cap \Gamma_{n-1}.$
In other words, it is possible to make this construction with $\tilde{M}_n$ verifying the following  property:   for every $1\leq k\leq l_{n+1}$ the number of $1\leq l\leq l_{n+1}$  such that $\tilde{M}_n(\cdot,l)=\tilde{M}_n(\cdot,k)$ is not grater than $$(\tilde{M}_n(2,k),\cdots, \tilde{M}_n(l_n-1,k),\tilde{M}_n(l_n,k)-|\bigcup_{g\in R_{n-1}}F_{n}\setminus F_{n}g^{-1}\cap \Gamma_{n-1}|)!$$
Among the hypothesis of  Proposition \ref{construction}, we ask a stronger condition on the number of columns of $M_n$ which  is stable under multiplication of matrices,  unlike the condition that we mention in this remark.
}
\end{remark}



\section{Characterization of Choquet simplices}

A compact, convex, and metrizable subset $K$ of a locally convex
real vector space is said to be a (metrizable) Choquet simplex, if
for each $v\in K$ there is a unique probability measure $\mu$
supported on the set of extreme points of $K$ such that $\int x
d\mu(x)=v$.

\medskip

 In this section we show that any metrizable
Choquet simplex is affine homeomorphic to the inverse limit defined by a managed sequence
of matrices satisfying the additional restriction on the number of columns.    




\subsection{Finite dimensional Choquet simplices}
For technical reasons, we have to separate the  finite and the infinite dimensional cases.

\begin{lemma}\label{lemma-inverse-limit-1} Let $K$ be a finite
dimensional metrizable Choquet simplex with exactly $d\geq 1$
extreme points. Let $(p_n)_{n\geq 0}$ be an increasing sequence of
positive integers such that for every $n\geq 0$ the integer $p_n$
divides $p_{n+1}$, and let $k\geq \max\{2,d\}$. Then   there exist
an increasing  subsequence $(n_i)_{i\geq 0}$ of indices and a
sequence $(M_i)_{i\geq 0}$ of square $k$-dimensional matrices
which is managed by $(p_{n_i})_{i\geq 0}$  such
that $K$ is affine homeomorphic to
$\varprojlim_n(\triangle(k,p_{n_i}), M_i)$.
\end{lemma}

\begin{proof} Let $k\geq \max\{3,d\}$, we will define the subsequence  $(n_i)_{i\geq 0}$  by induction on $i$
through a condition explained later. For every $i\geq 0$, we define $M_i$ the $k$-dimensional matrix by
 $$
M_i(l,j)=\left\{\begin{array}{lll}
  \frac{p_{n_{i+1}}}{p_{n_i}}-k(k-1) & \mbox{ if } & 1\leq  l=j \leq d\\
  k & \mbox{ if  } & l\neq j, 1\leq l\leq k \mbox{ and } 1\leq  j \leq d \\
  M_i(l,d) & \mbox{ if } & d< j\leq k.
\end{array}\right.$$
We always suppose that  $n_{i+1}$ is sufficiently large in order to have $\frac{p_{n_{i+1}}}{p_{n_i}}-k(k-1)>0$.

By the very definition,   $M_i$ is a positive matrix having $k\geq
3$   rows and columns; $\sum_{l=1}^k
M_i(l,j)=\frac{p_{n_{i+1}}}{p_{n_i}}$ for every $1\leq j\leq k$
and the range of $M_{i}$ is at most $d$. Thus the convex set
$\varprojlim_n(\triangle(k,p_{n_i}), M_i) $ has at most  $d$
extreme points.

If it has  exactly $d$ extreme points,  it is affine homeomorphic to $K$.
We will choose the sequence $(p_{n_i})_{i\geq 0}$ in order that  $P=\bigcap_{i\geq 0}M_0\cdots M_i(\triangle(k,p_{n_{i+1}}))$ 
has   $d$ extreme points, which implies that $\varprojlim_n(\triangle(k,p_{n_i}), M_i) $ has exactly $d$ extreme points.

For every $i\geq 0$,  the set  $P_i=M_0\cdots M_i(\triangle(k,p_{n_{i+1}}))$   is the closed convex set generated by the vectors  $v_{i,1},\cdots , v_{i,d}$, where
$$
v_{i,l}=\frac{1}{p_{n_{i+1}}}M_0\cdots M_i(\cdot,l), \mbox{ for every } 1\leq l\leq d.
$$
Since every $v_{i,l}$ is in  $\triangle(k,p_{n_0})$, there exists a sequence $(i_j)_{j\geq 0}$ such that for every $1\leq l\leq d$, the sequence $(v_{i_j,l})_{j\geq 0}$ converges to an element $v_l$ in   $\triangle(k,p_{n_0})$. Observe that $P$ is the closed convex set generated by $v_1,\cdots, v_d$. Thus if $v_1,\cdots, v_d$ are linearly independent then $P$ has $d$ extreme points.

Since for every $1\leq l\leq d$ we have 
$
\sum_{j=1}^k\frac{1}{p_{n_{i+1}}}M_0\cdots M_i(j,l)=\frac{1}{p_{n_0}},
$
there exists a positive vector $\delta_l^{(i)}=(\delta_{1,l}^{(i)},\cdots, \delta_{k,l}^{(i)})^T$ such that $\sum_{j=1}^k\delta_{j,l}^{(i)}=1$ and such that   for each $1\leq j\leq k$ 
$$
\frac{1}{p_{n_{i+1}}}M_0\cdots M_i(j,l)=\delta_{j,l}^{(i)}\frac{1}{p_{n_0}}.
$$
Thus if $B_i$ is the matrix given by
$$
B_i(\cdot,l)=\left\{ \begin{array}{lll}
                      v_{i,l} & \mbox{ if } 1\leq l\leq d\\
                       \frac{1}{p_{n_0}}e_l^{(k)} & \mbox{ if } d+1\leq l\leq k.
                  \end{array}\right.,
$$
then $B_i=DA_i$, where $D$ is the $k$-dimensional diagonal matrix given by 
$$D_i(l,l)=\frac{1}{p_{n_0}}, \mbox{ for every } 1\leq l\leq k,$$
and $A_i$ is the $k$-dimensional matrix defined by 
$$
A_i(\cdot ,l)=\left\{\begin{array}{lll}
               \delta_{l}^{(i)} & \mbox{ if } & 1\leq l \leq d\\
               e_l^{(k)}  &\mbox{ if } & d+1\leq l\leq k.
\end{array}\right..
$$
 If $\lim_{j\to\infty}A_{�_j}=A$ is invertible ($A$ is the $k$-dimensional matrix whose columns are the vectors $\lim_{j\to \infty}\delta_l^{(i_j)}$ and the canonical vectors  $e_{d+1}^{(k)},\cdots, e_k^{(k)}$), then $v_1, \cdots, v_l$ are linearly independent. For this it is enough to show that $A$ is strictly diagonally dominant (see the  Levy-Desplanques Theorem in \cite{HJ}).


Now we will define $(n_i)_{i\geq 0}$ in order that $A$ is  strictly diagonally dominant. 

Let $\varepsilon\in (0,\frac{1}{4})$. Let $n_0=0$ and $n_1>n_0$ such that for every $1\leq l\leq d$,
$$
\delta_{l,l}^{(0)}=1-\frac{p_{n_0}}{p_{n_1}} \sum_{j=1, j\neq l}^kM_0(j,l) =1-\frac{p_{n_0}}{p_{n_1}}k(k-1)\geq \frac{3}{4}+\varepsilon.
$$
For $i\geq 1$ we choose $n_{i+1}>n_i$ in order that
$$
\frac{1}{p_{n_{i+1}}}M_0\cdots M_{i-1}(l,l)<\varepsilon\frac{1}{p_{n_0}k(k-1)2^i}, \mbox{ for every } 1\leq l\leq d.
$$
After a standart computation, for every $i\geq 1$ and $1\leq l\leq d$ we get
$$
\delta_{l,l}^{(i)}\geq \delta_{l,l}^{(i-1)}-\frac{p_{n_0}}{p_{n_{i+1}}}k(k-1)M_0\cdots M_{i-1}(l,l),
$$
which implies that
$$
\delta_{l,l}^{(i)}\geq \delta_{l,l}^{(0)}-\varepsilon\sum_{j\geq 1}\frac{1}{2^j}\geq \frac{3}{4}.
$$
It follows that   $A(l,l)\geq \frac{3}{4}$ for every $1\leq l\leq k$, and since the sum of the elements in a column of $A$ is equal to $1$, we deduce that $A$ is  strictly diagonally dominant.
\end{proof}

\subsection{Infinite dimensional Choquet simplices}
We use the following characterization of infinite dimensional metrizable Choquet simplex.
\begin{lemma}[\cite{LL}, Corollary p.186]\label{lemma-LL}
For every infinite dimensional metrizable Choquet simplex $K$, there exists a sequence of
matrices $(A_n)_{n\geq 1}$ such that for every $n\geq 1$
\begin{enumerate}
\item $A_n(\triangle(n+1,1))=\triangle(n,1)$,  
\item $K$ is affine homeomorphic to
$\varprojlim_n(\triangle(n,1),A_n)$.
\end{enumerate}
\end{lemma}

Our strategy is to approximate the sequence of matrices $(A_n)_n$ by a managed sequence.
Then we show that the associated  inverse limits are affine homeomorphic. For this, we need the following classical  density result, whose proof follows from the fact that every non cyclic subgroup of $\RR$ is dense.

\begin{lemma}\label{density}
Let ${\bf r}=(r_n)_{n\geq 0}$ be a sequence of integers such that
$r_n\geq 2$ for every $n\geq 0$. Let $C_{\bf r}$ be the subgroup
of $(\RR,+)$ generated by $\{(r_0\cdots r_n)^{-1} : n\geq 0\}$. Then
 $$(C_{\bf r})^p\cap\triangle(p,1)\cap\{v\in \RR^p: v>0\}$$ is dense in $\triangle(p,1)$, for every $p\geq 2$, where $(C_{\bf r})^p$ is the Cartesian product  $\prod_{i=1}^pC_{\bf r}$.
\end{lemma}

\begin{lemma}\label{infinite} 
Let $K$ be an infinite dimensional metrizable Choquet simplex, and
let $(p_n)_{n\geq 0}$ be an increasing sequence of positive
integers such that for every $n\geq 0$ the integer $p_n$ divides
$p_{n+1}$.   Then there exist an increasing  subsequence $(n_i)_{i\geq 1}$
of indices  and a sequence of matrices  $(M_i)_{i \geq 1}$  managed by  $(p_{n_i})_{i\geq 0}$ such that for every $i\geq 0$,   $$k_{i+1}\leq \min\{M_i(l,k): 1\leq l\leq k_i, 1\leq k\leq k_{i+1}\}, $$ and such that
$K$ is affine homeomorphic to the inverse limit
$\varprojlim_n(\triangle(k_i, p_{n_i}), M_i)$, where $k_i$ is the
number of rows of $M_i$, for every $i\geq 0$.
\end{lemma}
\begin{proof}
For every $n\geq 0$, let $r_n\geq 2$ be the integer such that
$p_{n+1}=p_nr_n$.

Let $(A_n)_{n\geq 1}$ be the sequence of matrices given in Lemma
\ref{lemma-LL}. We can assume that
$
A_{n}: \triangle(n+3,1)\longrightarrow \triangle(n+2,1),
$
for every $n\geq 1$. Now we define the subsequence $(n_{i})_{i}$ by induction.
\medskip

We set $n_1=0$.

\medskip

Let $i\geq 1$ and suppose that we have defined $n_{i}\geq 0$. We set
${\bf r^{(i)}}=(r_n)_{n\geq n_i}$. For every $1\leq j\leq i+3$,
Lemma \ref{density} ensures the existence of $v^{(i,j)}\in
(C_{{\bf r^{(i)}}})^{i+2}\cap \triangle(i+2,1)\cap\{v\in \RR^{i+2}:
v>0\}$ such that
\begin{equation}\label{infinite-eq1}
\|v^{(i,j)}-A_{i}(\cdot, j)\|_1<\frac{1}{2^i}.
\end{equation}

Let $B_i$ be the matrix given by
$$
B_i(\cdot,j)=v^{(i,j)}, \mbox{ for every } 1\leq j\leq i+3.
$$
Observe that (\ref{infinite-eq1})  implies that
$$
\sum_{n\geq 1}\sup\{\|A_{n}v-B_{n}v\|_1: v\in
\triangle_{n+3}\}<\infty.
$$
It follows from \cite[Lemma 9]{CRL2} that $K$ is affine
homeomorphic $\varprojlim_n(\triangle(i+2,1), B_i)$.

\medskip

Let $n_{i+1}>n_i$ be such that $r_{n_i}\cdots
r_{n_{i+1}-1}v^{(i,j)}$ is an integer vector and such that
$r_{n_i}\cdots r_{n_{i+1}-1}v^{(i,j)}>i+3$, for every $1\leq j\leq
i+3$.

We define
$$
M_i=\frac{p_{n_{i+1}}}{p_{n_i}}B_i.
$$

Thus $M_i=P_i^{-1}B_iP_{i+1}$, where $P_i$ is the diagonal matrix
given by $P_i(j,j)=p_{n_i}$ for every $1\leq j\leq i+2$ and $i\geq
1$. This shows that $\varprojlim_n(\triangle(i+2,1), B_i)$ is
affine homeomorphic to $\varprojlim_n(\triangle(i+2,p_{n_{i}}),
M_i)$.

The proof conclude verifying that $(M_i)_{i\geq 0}$ is managed by $(p_{n_i})_{i\geq 0}$.
\end{proof}

\section{Proof of the main theorems. }

\subsection{Proof of Theorem A}
The proof of Theorem A is a corollary of previous results.
\begin{proof}[Proof of Theorem A]
 Let $ext(K)$ be the set of extreme points of
$K$.  If $ext(K)$ is finite, then the proof is direct from  Proposition \ref{construction} and Lemma
\ref{lemma-inverse-limit-1}. If
$ext(K)$ is infinite, the proof follows from  Proposition \ref{construction} and Lemma \ref{infinite}.
\end{proof}

\subsection{Proof of Theorem B}

We refer to \cite{Do2} for definitions and properties about Toeplitz $\ZZ$-subshifts or Toeplitz flows. See \cite{DHS} and \cite{HPS}  for details about ordered Bratteli
diagram, Kakutani-Rokhlin partitions and dimension groups associated to minimal $\ZZ$-actions on the Cantor set.

We denote by $\Sigma$ a finite alphabet with at least two elements. For $x=(x_n)_{n\in\ZZ}\in \Sigma^{\ZZ}$ and $n\leq m\in\ZZ$, we set $x[n,m]=x_n\cdots x_m$. In a similar way, if $w=w_0\cdots w_{n-1}$ is a word in $\Sigma^n$, we set $w[k,l]=w_k\cdots w_l$ for every $0\leq k\leq l<n$.

 The next result follows from the proof of  \cite[Theorem 8]{GJ}.

\begin{lemma}\label{lemma-KR}
Let $x_0\in\Sigma^{\ZZ}$ be a Toeplitz sequence and let $(X,\sigma|_X,\ZZ)$ be the associated Toeplitz $\ZZ$-subshift. There exist a period structure $(p_n)_{n\geq 0}$  of $x_0$  and a sequence of matrices $(A_n)_{n\geq 0}$ managed by $(p_n)_{n\geq 0}$ such that the dimension group associated to $(X,\sigma|_{X},\ZZ)$ is isomorphic to
$$
\xymatrix{
{\ZZ^{}}\ar@{->}[r]^{A_0^T} & {\ZZ^{k_1}}\ar@{->}[r]^{A^T_1} &
{\ZZ^{k_2}}\ar@{->}[r]^{A^T_2}& {\cdots}}.
$$
Furthermore, if $k_n$ is the number of rows of $A_n$ and $r_n=\frac{p_{n+1}}{p_n}$, then for every $m>n>0$ and $1\leq k\leq k_{m}$,  
\begin{eqnarray*}
 |\{1\leq l\leq k_{m}: A_{n,m-1}(\cdot, l)=A_{n, m-1}(\cdot,k)\}| & &\\  
 \leq  (A_{n,m-1}(1,k)-r_{n+2}\cdots r_{m-1}, \cdots,  A_{n, m-1}(k_n,k)-r_{n+2}\cdots r_{m-1})!, & &
 \end{eqnarray*}
 where  $A_{n,m-1}=A_n\cdots A_{m-1}$.  
  \end{lemma} 

\begin{proof}
In the proof of Theorem 8 in \cite{GJ} the authors show there exist a period structure $(p_n)_{n\geq 1}$  of $x_0$  and a sequence $(\P_n)_{n\geq 0}$  of nested Kakutani-Rokhlin partitions of $(X,\sigma|_X,\ZZ)$ such that $\P_0=\{X\}$ and $\P_n=\{T^j(C_{n,k}): 0\leq j< p_n, 1\leq k\leq k_n\}$, where 
$$
C_{n,k}=\{x\in X: x[0,p_n-1]=w_{n,k} \} \mbox{ for every } 1\leq k\leq k_n,
$$  
with $W_n=\{w_{n,1}, \cdots, w_{n,k_n}\}$ the set of the words $w$  of $x_0$ of length $p_n$ verifying $w[0,p_{n-1}-1]=x_0[0,p_{n-1}-1]$, for every $n\geq 1$ (with $p_0=1$).

Thus the dimension group with unit associated to $(X,\sigma|_X,\ZZ)$ is isomorphic to
$$
\lim_{\rightarrow n}(\ZZ^{k_n},A_n^T)=\xymatrix{
{\ZZ^{}}\ar@{->}[r]^{A_0^T} & {\ZZ^{k_1}}\ar@{->}[r]^{A_1^T} &
{\ZZ^{k_2}}\ar@{->}[r]^{A_2^T}& {\cdots}},
$$
where $A_n(i,j)$ is the number of times that the word $w_{n,i}$ appears in the word $w_{n+1,j}$, for every $1\leq i\leq k_{n}$, $1\leq j\leq k_{n+1}$ and $n\geq 1$, and the matrix $A_0^T$ is the vector in $\ZZ^{k_1}$ whose coordinates are equal to $p_1$. 

Since $w_{n+1,i}\neq w_{n+1,j}$ for $i\neq j$,  equal columns of the matrix $A_n$   produce different concatenations of  words in $W_{n}$. This implies that for every $1\leq k\leq k_{n+1}$, the number of columns of $A_n$ which are equal to $A_n(\cdot, k)$ can not exceed the number of different concatenations of $r_n$ words in $W_n$ using exactly $A_n(j,k)$ copies of $w_{n,j}$, for every $1\leq j\leq k_n$. This means that the number of columns which are equal to $A_n(\cdot, k)$ is smaller or equal to $(A_n(1,k), \cdots,  A_n(k_n,k))!.$ 

Now fix $n>0$ and take $m>n$.  The coordinate $(i,j)$ of the matrix $A_{n, m-1}$ contains the number of times that the word $w_{n,i}\in W_n$ appears in $w_{m,j}\in W_m$.  Observe that every word $u$ in $W_m$ is a concatenation of $r_{n+2}\cdots r_{m-1}$ words in $W_{n+2}$. In addition, each word in $W_{n+2}$ starts with  $x_0[0,p_{n+1}-1]\in W_{n+1}$, which is a word   containing every word in $W_n$ (we can always assume that the matrices $A_n$ are positive).  Thus there exist $0\leq l_1<\cdots<l_{r_{n+1}\cdots r_{m-1}}<p_m$ such that $u[l_s,l_s+p_n-1]=w[l_s,l_s+p_{n}-1] \in W_n$, for every $1\leq s\leq r_{n+2}\cdots r_{m-1}$ and  $u,w\in W_m$.

 This implies that the number of all possible concatenations of words in $W_n$   producing a word in $W_m$  according to the column $k$ of the matrix $A_{n,m-1}$ is smaller or equal to
  $$
   (A_{n, m-1}(1,k)-r_{n+2}\cdots r_{m-1}, \cdots, A_{n,m-1}(k_n, k)-r_{n+2}\cdots r_{m-1})!.
  $$
\end{proof}

\begin{proof}[Proof of Theorem B]
Let $x_0\in X$ be a Toeplitz sequence.
Let $(p_n)_{n\geq 1}$ and $(A_n)_{n\geq 0}$ be the period structure of $x_0$ and the sequence of matrices  given by Lemma \ref{lemma-KR} respectively. It is straightforward to check that Lemma \ref{lemma-KR} is also true if we take a subsequence of $(p_n)_{n\geq 0}$. Thus we can assume that for every $n\geq 1$, the matrix $A_n$ has its coordinates strictly grater than $1$ and that there exist positive integers $r_{n,1}, \cdots, r_{n,d}>1$ such that
  $$
  \frac{p_{n+1}}{p_n}=r_n=r_{n,1}\cdots r_{n,d}.
  $$
  Le define $q_{n+1,i}=r_{0,i}\cdots r_{n,i}$ for every $1\leq i\leq d$, and $\Gamma_{n+1}=\prod_{i=1}^d q_{n+1,i}\ZZ$, for every $n\geq 0$. We have $\Gamma_{n+1}\subseteq \Gamma_n$, $\bigcap_{n\geq 1}\Gamma_n=\{0\}$ and $|\ZZ^d/\Gamma_n|=p_{n}$. Let $(F_n)_{n\geq 0}$ be a F\o lner sequence associated to $(\Gamma_n)_{n\geq 1}$ as in Lemma \ref{folner2}.  We denote $R_n$ as in Section \ref{realization-ordered-groups} (the set that defines "border"). 
   
 Now, we define an increasing sequence $(n_i)_{i\geq 1}$   of integers as follows:
 
 We set $n_1=1$.  For $i\geq 1$, given $n_i$ we chose $n_{i+1}>n_i+1$ such that
 $$
 \sum_{g\in R_{n_i}}\frac{|F_{n_{i+1}}\setminus F_{n_{i+1}}-g|}{|F_{n_{i+1}}|}<\frac{1}{|F_{n_i}|r_{n_i}r_{n_{i}+1}}.
$$
Thus we have
\begin{eqnarray*}
\frac{|F_{n_{i+1}}|}{|F_{n_i}|}-\sum_{g\in R_{n_i}}|F_{n_{i+1}}\setminus F_{n_{i+1}}-g|&>& \frac{|F_{n_{i+1}}|}{|F_{n_i}|}-\frac{|F_{n_{i+1}}|}{|F_{n_i}|r_{n_i}r_{n_{i}+1}}\\
&=& r_{n_i}\cdots r_{n_{i+1}-1}-r_{n_i+2}\cdots r_{n_{i+1}-1}\\
&>& r_{n_i}\cdots r_{n_{i+1}-1}-k_{n_i}r_{n_i+2}\cdots r_{n_{i+1}-1}
\end{eqnarray*}
 Let $M_0=A_0$ and  $M_i=A_{n_i}\cdots A_{n_{i+1}-1}$ be for every $i\geq 1$.  For every $1\leq k\leq k_{n_{i+1}}$ we get
$$
M_i(k_{n_i},k)- \sum_{g\in R_{n_i}}|F_{n_{i+1}}\setminus F_{n_{i+1}}-g|> M_i(k_{n_i},k)-r_{n_i+2}\cdots r_{n_{i+1}-1},
$$
which implies that
$$
(M_i(1,k),\cdots, M_i(k_{n_i}-1,k),M_i(k_{n_i},k)-\sum_{g\in R_{n_i}}|F_{n_{i+1}}\setminus F_{n_{i+1}}-g|)! 
$$
is grater than
$$
(M_i(1,k)-r_{n_i+2}\cdots r_{n_{i+1}-1}, \cdots, M_i(k_{n_i},k)-r_{n_i+2}\cdots r_{n_{i+1}-1})!
$$  
Then from the previous inequality  and  Lemma \ref{lemma-KR} we get that the number of columns of $M_i$ which are equal to  $M_ i(\cdot, k)$  is smaller than
$$
(M_i(1,k),\cdots, M_i(k_{n_i}-1,k),M_i(k_{n_i},k)-\sum_{g\in R_{n_i}}|F_{n_{i+1}}\setminus F_{n_{i+1}}-g|)! 
$$
As in the proof of Proposition \ref{construction}, we define $\tilde{M}_i$ and we call $l_i$ and $l_{i+1}$ the number of rows and columns of $\tilde{M}_i$ respectively, for every $i\geq 0$. According to the notations of the proof of Proposition \ref{construction}, in our case $M_0$ corresponds to the matrix $M$ and $\tilde{M}_0$ corresponds to the matrix $\tilde{M}$.   Observe  that the bound on the number of columns which are equal to $M_i(\cdot,k)$ (and then to $\tilde{M}_i(\cdot,k)$) ensures the existence of enough possibilities to fill the coordinates of $F_{n_i}$ in order to obtain different functions $B_{i,1}\cdots, B_{i,l_i}\in \{1,\cdots, l_1\}^{F_{n_i}}$  as in the proof of Proposition \ref{construction}, for every $i\geq 1$ (see Remark \ref{remark-bound}).  

Lemma \ref{sufficient-condition} implies that the Toeplitz $\ZZ^d$-subshift $(Y,\sigma|_Y,\ZZ^d)$ defined from $(B_{i,1},\cdots, B_{i,l_i})_{i\geq 1}$ has an ordered group $\G(Y,\sigma|_Y,\ZZ^d)$ isomorphic to  
$(H/inf(H), (H/inf(H))^+,u+inf(H))$, where $(H,H^+)$ is given by 
$$
\xymatrix{{\ZZ^{}}\ar@{->}[r]^{\tilde{M}_0^T} &{\ZZ^{l_0}}\ar@{->}[r]^{\tilde{M}^T_1} & {\ZZ^{l_2}}\ar@{->}[r]^{\tilde{M}_2^T} &
{\ZZ^{l_3}}\ar@{->}[r]^{\tilde{M}^T_3}& {\cdots}},
$$
with $\tilde{M}_0=|F_1|(1,\cdots,1)$ and $u=[1,0]$.

Lemma \ref{nexo2} implies that $(H,H^+,u)$ is isomorphic to the dimension group with unit 
$(J,J^+,w)$ associated to $(X,\sigma|_X,\ZZ)$. Thus $(J/inf(J), (J/inf(J))^+,w+inf(J))$, the ordered group  associated to $(X,\sigma|_X,\ZZ)$,  is isomorphic to  $\G(Y,\sigma|_{Y},\ZZ^d)$.  We conclude the proof applying Theorem \ref{theorem1}.
  \end{proof}

 In \cite{Su},  the author shows that
every minimal Cantor system $(Y,T,\ZZ)$ having an associated
Bratteli diagram which satisfies the equal path number property, is strong
orbit equivalent to a Toeplitz subshift $(X,\sigma|_X,\ZZ)$.  Thus the next result is inmediat.

\begin{corollary}\label{divisible}
Let $(X,T,\ZZ)$ be a minimal Cantor having an associated
Bratteli diagram which satisfies the equal path number property. Then for every $d\geq 1$ there exists a  Toeplitz subshift $(Y,\sigma|_Y,\ZZ^d)$ which is orbit equivalent to
$(X,T,\ZZ)$.
\end{corollary}

\medskip

{\bf Acknowledgments.}  We would like to  thank the referee for many
valuable comments which served to improve the article.

\bibliographystyle{alpha}

\end{document}